\documentclass[11pt]{article}

\usepackage{amsmath,amssymb,amsthm,mathtools,bm}
\usepackage{geometry}
\usepackage{enumitem}
\usepackage{mathrsfs}
\usepackage{graphicx}
\usepackage{placeins}
\usepackage{float}
\usepackage{booktabs}
\usepackage[sort,nocompress,noadjust,space]{cite}
\usepackage{hyperref}
\hypersetup{hidelinks, pdfkeywords={planar conductivity problem, multiple inclusions, Faber-Walsh polynomials, generalized polarization tensors, Neumann-Poincare operator}}
\newcommand{\RR}{\mathbb{R}}
\newcommand{\CC}{\mathbb{C}}
\newcommand{\ZZ}{\mathbb{Z}}
\newcommand{\NN}{\mathbb{N}}
\newcommand{\p}{\partial}
\newcommand{\Kcal}{\mathcal{K}}
\newcommand{\Scal}{\mathcal{S}}
\newcommand{\Lcal}{\mathcal{L}}
\newcommand{\Hcal}{\mathcal{H}}
\DeclareMathOperator{\RePart}{Re}
\DeclareMathOperator{\ImPart}{Im}

\newtheorem{theorem}{Theorem}[section]
\newtheorem{lemma}[theorem]{Lemma}
\newtheorem{proposition}[theorem]{Proposition}
\newtheorem{corollary}[theorem]{Corollary}
\theoremstyle{definition}
\newtheorem{definition}[theorem]{Definition}
\newtheorem{remark}[theorem]{Remark}

\title{Faber--Walsh field expansions for the planar conductivity problem with multiple inclusions}
\author{
Doosung Choi\thanks{Department of Mathematics, Applied Mathematics, and Statistics, Case Western Reserve University, Cleveland, OH 44106, USA (doosung.choi@case.edu)}
\and
Mikyoung Lim\thanks{Department of Mathematical Sciences, Korea Advanced Institute of Science and Technology, 291 Daehak-ro, Yuseong-gu, Daejeon 34141, Republic of Korea (mklim@kaist.ac.kr)}
}
\date{}

\begin{document}
\maketitle

\begin{abstract}
We derive an explicit Faber--Walsh representation for the planar conductivity problem with finitely many disjoint inclusions of general shape and independently chosen positive conductivities. Walsh's lemniscatic map gives a conformal coordinate without smallness or wide-separation assumptions. The field expansion is expressed through Faber--Walsh polarization tensors, which we relate to the complex contracted generalized polarization tensors by a triangular change of basis. For analytic interfaces, we prove conformal continuation and absolute and uniform convergence of the Walsh--Grunsky series near the boundary to obtain explicit factorizations of the Neumann--Poincar\'e operator and the Faber--Walsh polarization tensors. These factorizations separate canonical interaction, which can remain nonzero even when all Grunsky coefficients vanish, from conformal deformation, while retaining the componentwise conductivity dependence in a coupled resolvent. For harmonic polynomial incident fields, we prove geometric convergence of outgoing truncations with exact coefficients on compact subsets of fixed exterior branch domains. Numerical comparisons with independent boundary-integral solutions illustrate the representation for heterogeneous, asymmetric, and nearly touching configurations.
\end{abstract}



\section{Introduction}
\label{sec:intro}

The planar conductivity problem admits a boundary-integral formulation
in terms of the Neumann--Poincar\'e (NP) operator. The coefficients in
the far-field multipole expansion of the field perturbation, called generalized
polarization tensors (GPTs), contain both geometric and material
information. They have been used for reconstruction and shape description
\cite{AmmariGarnierKangLimYu2014,AmmariKang2004,AmmariKang2007,
AmmariKangLimZribi2012}, target identification
\cite{AmmariBoulierGarnierJingKangWang2014}, and tracking of moving targets
\cite{AmmariBoulierGarnierKangWang2013}. For several inclusions, GPTs also
record interactions between the interfaces \cite{AmmariKangKimLim2005}.
In two dimensions, combining the GPTs using complex monomials gives the
complex contracted generalized polarization tensors (CGPTs), which are
particularly convenient for conformal-mapping methods. For a simply
connected inclusion, the CGPTs determine coefficients of the exterior
Riemann map and give non-iterative reconstruction formulas
\cite{ChoiHelsingKangLim2023,ChoiKimLim2021Shape,KangLeeLim2015}.
Replacing the complex monomials by Faber polynomials leads to the Faber
polynomial polarization tensors (FPTs). In this basis, the NP operator can
be represented through Grunsky coefficients, and the FPTs are the
coefficients of the geometric multipole expansion
\cite{ChoiKimLim2023,JungLim2020,JungLim2021}. This Faber-polynomial
formulation has also been extended to imperfect interfaces \cite{ChoiLim2024}.

GPTs and FPTs also arise in several design and sensing problems. Vanishing
low-order GPTs have been used to enhance near-cloaking
\cite{AmmariKangLeeLim2013}, while FPTs have been used to construct
semi-neutral inclusions of general shape \cite{ChoiKimLim2023}. Related
conformal expansions have been applied to effective properties of composites
and to neutral core--shell structures \cite{CherkaevKimLim2022,LimMilton2020}.
The NP operator also connects this framework to spectral problems arising in
near-field optics \cite{AmmariChoiYu2018}, plasmonic resonance
\cite{JungLim2023}, and electro-sensing \cite{BonnetierTrikiTsou2018}.

The Faber-based formulations described above are built around a single simply
connected inclusion or coated structures generated from one exterior conformal
map. For a union of $N\ge2$ disjoint inclusions, however, the exterior is
$N$-connected and there is no exterior-disk coordinate. The off-diagonal NP
blocks also couple the interfaces. These interactions have been studied from
several directions. Li and Vogelius \cite{LiVogelius2000} and Ammari et al.\
\cite{AmmariBonnetierTrikiVogelius2015} obtained gradient estimates for
composite media with several inclusions. For perfectly conducting or insulating
inclusions, Bao, Li, and Yin \cite{BaoLiYin2010} derived gradient bounds that
depend on the distances between inclusions. Kang et al.\
\cite{KangKimLeeShinYu2016} used a block NP system for inclusions with
different complex conductivities to study solvability and uniform estimates
under conditions on the conductivities.

There are also several approaches for computing the field in a multi-inclusion
geometry. Greengard and Moura \cite{GreengardMoura1994} combined boundary
integral equations with the fast multipole method and the generalized minimal
residual method for multiphase composites. Crowdy, Tanveer, and DeLillo
\cite{CrowdyTanveerDeLillo2016} used a hybrid Fourier--Laurent and
conformal-mapping basis for close-to-touching discs. Other series methods
include complex-potential expansions for multiple elliptical inhomogeneities
\cite{EbrahimiBalintDini2023} and equivalent-inclusion methods for arbitrarily
shaped inhomogeneities \cite{DingEtAl2025}. Our aim here is to obtain an exact
field expansion for general planar inclusions in one conformal coordinate,
while allowing the conductivity to differ from one component to another.

Walsh's lemniscatic mapping theorem \cite{Walsh1956} provides a conformal
coordinate for the multiply connected exterior, and Faber--Walsh (FW)
polynomials \cite{SeteLiesen2017,Walsh1958} provide the corresponding
polynomial basis. S\`ete \cite[Proposition~3.1(4)]{Sete2013} derived the
expansion of FW polynomials in the lemniscatic coordinate that we use below.
S\`ete and Liesen \cite{SeteLiesen2016} construct lemniscatic maps for special
geometries, while Nasser, Liesen, and S\`ete \cite{NasserLiesenSete2016}
compute them by a boundary-integral method. Schiefermayr and S\`ete
\cite{SchiefermayrSete2023} characterize Walsh weights through Green's
functions and determine weights and centers for polynomial pre-images.
Generalized Grunsky operators on multiply connected domains are studied in
\cite{RadnellSchippersStaubach2020}, while Grunsky coefficients have been used
to represent NP operators in \cite{ChoiLimShipman2025,JiKang2023RotSym}.
We use the reciprocal FW expansions directly in the coupled conductivity
problem. We do not identify their coefficient matrix with an abstract
generalized Grunsky operator.

With these tools, we consider $N$ disjoint simply connected inclusions with
positive background and inclusion conductivities. The conductivity may differ
from one inclusion to another, and we make no smallness or wide-separation
assumptions. The resulting moment coefficients are the Faber--Walsh
polarization tensors (FWPTs). For analytic interfaces,
Lemma~\ref{lem:analytic-boundary-traces} establishes the summable coefficient
bound \eqref{eq:Grunsky-decay}, which justifies boundary evaluation of the
reciprocal expansions and the regularized kernel and leads to the NP
factorization \eqref{eq:K-Grunsky} in Theorem~\ref{thm:K-Grunsky}. This formula
separates canonical interaction from the Walsh--Grunsky correction. Even when
the conformal map is the identity and all Grunsky coefficients vanish, the
canonical term need not be zero. In the lemniscatic coordinate, one polynomial
system serves for both incident fields and tested moments on all components,
and the same Walsh--Grunsky coefficient matrix appears in the moment, source,
and NP factorizations. The full boundary resolvent retains both inter-component
coupling and the componentwise material matrix, without assuming that the
material and NP matrices commute.

For $C^{1,\alpha}$ interfaces and harmonic polynomial incident fields,
Theorem~\ref{thm:FW-truncation} gives the geometric outgoing-truncation estimate
\eqref{eq:FW-truncation-estimate} with exact coefficients on compact subsets of
fixed exterior branch domains. A triangular change of polynomial basis relates
the FWPTs to the CGPTs, and the simply connected reduction recovers the
classical Faber--Grunsky formulas. We compare the FW fields with independent
Nystr\"om solutions for heterogeneous, asymmetric, and nearly touching
configurations.

Section~\ref{sec:NP} gives the boundary-integral formulation and the FW
coordinates. Section~\ref{sec:tensors} contains the field expansion, the
truncation estimate, and the boundary factorizations. Numerical examples and
concluding remarks follow in Sections~\ref{sec:numerics} and
\ref{sec:conclusion}.

\section{Formulation for Multiple Conductivity Inclusions}
\label{sec:NP}


\subsection{Boundary Integral and Neumann--Poincar\'e Formulation}

Identifying $x=(x_1,x_2)\in\RR^2$ with $z=x_1+\mathrm{i}x_2\in\CC$, we fix
$N\ge1$ and let $D_1,\ldots,D_N\subset\RR^2\simeq\CC$ be bounded
simply connected domains with $C^{1,\alpha}$ boundaries, $0<\alpha<1$,
whose closures are pairwise disjoint:
\begin{equation}
    \overline{D_i}\cap\overline{D_j}=\varnothing,
    \qquad 1\le i\ne j\le N.
    \label{eq:disjoint}
\end{equation}
We set $D:=\bigcup_{j=1}^N D_j$, $\Gamma_j:=\p D_j$, and
$\Gamma:=\bigcup_{j=1}^N\Gamma_j$. We write
$\widehat\CC:=\CC\cup\{\infty\}$ and
$\Omega:=\widehat\CC\setminus\overline D$. Thus, $\Omega$ is
$N$-connected, with infinity an interior point.

The background conductivity is $\sigma_m>0$, and the conductivity in
$D_j$ is a constant $\sigma_j>0$, with
\[
\sigma_j\ne\sigma_m,\qquad j=1,\ldots,N.
\]
The inclusion conductivities may differ from one another. The conductivity distribution is
\begin{equation}
    \sigma(x)=
    \begin{cases}
        \sigma_j, & x\in D_j,\quad j=1,\ldots,N,\\
        \sigma_m, & x\in\RR^2\setminus\overline D.
    \end{cases}
    \label{eq:sigma}
\end{equation}
For a real-valued function $H$ harmonic in all of $\RR^2$, we consider
\begin{equation}
    \begin{cases}
        \nabla\cdot(\sigma\nabla u)=0,
        & \text{in }\RR^2,\\[3pt]
        u|_{\Gamma_j}^{+}=u|_{\Gamma_j}^{-},
        & j=1,\ldots,N,\\[3pt]
        \sigma_m\p_{\nu_j}u|_{\Gamma_j}^{+}
        =\sigma_j\p_{\nu_j}u|_{\Gamma_j}^{-},
        & j=1,\ldots,N,\\[3pt]
        u(x)-H(x)=O(|x|^{-1}),
        & |x|\to\infty.
    \end{cases}
    \label{eq:conductivity}
\end{equation}
Here, $\nu_j$ points outward from $D_j$, the signs $+$ and $-$ denote
exterior and interior limits, and $\p_{\nu_j}$ is differentiation in
the direction $\nu_j$. Multiplying all conductivities by the same
positive constant leaves this potential problem unchanged.

For $\varphi_j\in H^{-1/2}(\Gamma_j)$, we define
\begin{equation}
    \Scal_{\Gamma_j}[\varphi_j](x)
    :=\frac{1}{2\pi}\int_{\Gamma_j}
       \log|x-y|\,\varphi_j(y)\,ds(y).
    \label{eq:SL}
\end{equation}
Here, $ds$ is arc-length measure. We use the Laplace single-layer
potential and seek
\begin{equation}
    u=H+\sum_{j=1}^N\Scal_{\Gamma_j}[\varphi_j].
    \label{eq:SLansatz}
\end{equation}
For $i,j\in\{1,\ldots,N\}$ and $x\in\Gamma_i$, we define
\begin{equation}
    \Kcal_{i j}^{*}[\varphi_j](x)
    :=\frac{1}{2\pi}\int_{\Gamma_j}
    \frac{(x-y)\cdot\nu_i(x)}{|x-y|^2}
    \varphi_j(y)\,ds(y),
    \label{eq:Kij}
\end{equation}
with principal value when $i=j$. The jump relation is
\begin{equation}
    \p_{\nu_j}\Scal_{\Gamma_j}[\varphi_j]\big|_\pm
    =\left(\pm\tfrac12 I+\Kcal_{jj}^{*}\right)\varphi_j.
    \label{eq:jump}
\end{equation}
Cross potentials have no jump. Consequently,
\[
    \p_{\nu_i}u\big|_\pm
    =\p_{\nu_i}H
      +\left(\pm\tfrac12 I+\Kcal_{ii}^{*}\right)\varphi_i
      +\sum_{\substack{j=1\\j\ne i}}^N
       \Kcal_{i j}^{*}\varphi_j.
\]
Substitution into the flux condition gives
\[
    \frac{\sigma_i+\sigma_m}{2}\varphi_i
    =(\sigma_i-\sigma_m)
      \left(\p_{\nu_i}H+
      \sum_{j=1}^N\Kcal_{i j}^{*}\varphi_j\right).
\]
We define the component contrasts and their vector by
\begin{equation}
    \lambda_j:=\frac{\sigma_j+\sigma_m}{2(\sigma_j-\sigma_m)},
    \qquad \boldsymbol\lambda:=(\lambda_1,\ldots,\lambda_N),
    \qquad |\lambda_j|>\frac12.
    \label{eq:lambda}
\end{equation}
The coupled equations are
\begin{equation}
    (\lambda_i I-\Kcal_{ii}^{*})\varphi_i
    -\sum_{\substack{j=1\\j\ne i}}^N
        \Kcal_{i j}^{*}\varphi_j
    =\p_{\nu_i}H,
    \qquad i=1,\ldots,N.
    \label{eq:block1}
\end{equation}
We introduce
\begin{equation}
    \mathcal K_\Gamma^*:=(\Kcal_{ij}^{*})_{i,j=1}^N,\quad
    \Lambda_\Gamma:=\begin{pmatrix}
        \lambda_1I_{\Gamma_1}&0&0\\
        0&\ddots&0\\
        0&0&\lambda_NI_{\Gamma_N}
    \end{pmatrix},\quad
    \boldsymbol\varphi:=(\varphi_j)_{j=1}^N,\quad
    \boldsymbol g_H:=(\p_{\nu_j}H)_{j=1}^N.
    \label{eq:blockK}
\end{equation}
The operators $\mathcal K_\Gamma^*$ and $\Lambda_\Gamma$ have
$N\times N$ boundary-operator blocks, and $\boldsymbol\varphi$ and
$\boldsymbol g_H$ are $N\times1$ boundary vectors.
The block $I_{\Gamma_j}$ is the identity on the $j$th boundary space.
The system becomes
\begin{equation}
    (\Lambda_\Gamma-\mathcal K_\Gamma^*)\boldsymbol\varphi
    =\boldsymbol g_H,
    \qquad
    \boldsymbol\varphi
    =(\Lambda_\Gamma-\mathcal K_\Gamma^*)^{-1}\boldsymbol g_H.
    \label{eq:blockres}
\end{equation}
For equal inclusion conductivities $\sigma_j=\sigma_0$, this reduces
to the scalar shift $\lambda I-\mathcal K_\Gamma^*$ with
$\lambda=(\sigma_0+\sigma_m)/(2(\sigma_0-\sigma_m))$.

We work on the componentwise zero-charge space
\begin{equation}
    \Hcal^{-1/2}_0(\Gamma)
    :=\prod_{j=1}^N H^{-1/2}_0(\Gamma_j),
    \qquad
    H^{-1/2}_0(\Gamma_j)
    :=\left\{\varphi_j\in H^{-1/2}(\Gamma_j):
       \int_{\Gamma_j}\varphi_j\,ds=0\right\}.
    \label{eq:charge-space}
\end{equation}
Integrals against constants are understood as trace duality pairings.

\begin{proposition}[Invertibility of the transmission operator]
\label{prop:material-invertibility}

Under the assumptions above,
$\Lambda_\Gamma-\mathcal K_\Gamma^*$
is invertible on
$\Hcal^{-1/2}_0(\Gamma)$ and on
$\prod_{j=1}^N L^2_0(\Gamma_j)$.

\end{proposition}

\begin{proof}

The identity
\[
\int_{\Gamma_i}\Kcal_{ij}^*\varphi_j\,ds
=
\frac{\delta_{ij}}{2}
\int_{\Gamma_j}\varphi_j\,ds
\]
shows that $\mathcal K_\Gamma^*$ preserves componentwise zero charge.
Since the diagonal NP blocks are compact on $C^{1,\alpha}$ boundaries
and the off-diagonal blocks are compact by positive separation,
$\Lambda_\Gamma-\mathcal K_\Gamma^*$ is Fredholm of index zero.

It remains to prove injectivity. Suppose $(\Lambda_\Gamma-\mathcal K_\Gamma^*)\boldsymbol\varphi=0$
and set
\[
v=\sum_{j=1}^N\Scal_{\Gamma_j}[\varphi_j].
\]
Then $v$ satisfies the homogeneous transmission problem and
$v(x)=O(|x|^{-1})$ at infinity. Green's identity gives
\[
\sigma_m\int_{\RR^2\setminus\overline D}|\nabla v|^2\,dx
+
\sum_{j=1}^N
\sigma_j\int_{D_j}|\nabla v|^2\,dx
=0.
\]
Since all conductivities are positive, $v$ is constant in each region
and hence vanishes by continuity and decay at infinity. The jump relation
then gives $\varphi_j=0$ for all $j$.

Thus the operator is injective, and Fredholm theory gives invertibility.
The same argument applies on both spaces.
\end{proof}

\begin{lemma}[Componentwise zero charge]
\label{lem:zerocharge}
For $H$ harmonic in all of $\RR^2$, the right-hand side
$\boldsymbol g_H$ and the solution $\boldsymbol\varphi$ of
\eqref{eq:blockres} have zero mean on every $\Gamma_j$.
\end{lemma}

\begin{proof}
The divergence theorem gives
$\int_{\Gamma_j}\p_{\nu_j}H\,ds=\int_{D_j}\Delta H\,dx=0$.
The normal derivative jump and the flux condition give
\[
    \int_{\Gamma_j}\varphi_j\,ds
    =\left(\frac{\sigma_j}{\sigma_m}-1\right)
       \int_{\Gamma_j}\p_{\nu_j}u\big|_-\,ds
    =\left(\frac{\sigma_j}{\sigma_m}-1\right)
       \int_{D_j}\Delta u\,dx=0,
    \quad j=1,\ldots,N.
\]
Equivalently, we integrate \eqref{eq:block1} and use
$(\lambda_j-1/2)\int_{\Gamma_j}\varphi_j\,ds=0$.
\end{proof}

For the smooth incident fields used below,
$\boldsymbol g_H\in\prod_{j=1}^N L^2_0(\Gamma_j)$.
Proposition~\ref{prop:material-invertibility} and uniqueness imply
$\varphi_j\in L^2(\Gamma_j)\subset L^1(\Gamma_j)$, since each
boundary has finite length. Boundary moments below are therefore
ordinary integrals, or equivalently Sobolev duality pairings.

\paragraph{Block elimination.}
We fix $j$ and list the remaining component indices as
$\mathcal J_j=\{i_1,\ldots,i_{N-1}\}$, in increasing order. We set
$A_i=\lambda_i I_{\Gamma_i}-\Kcal_{ii}^*$ and
\[
\mathcal A_{\mathcal J_j}:=
\begin{pmatrix}
A_{i_1}&-\Kcal_{i_1i_2}^*&\cdots&-\Kcal_{i_1i_{N-1}}^*\\
-\Kcal_{i_2i_1}^*&A_{i_2}&\cdots&-\Kcal_{i_2i_{N-1}}^*\\
\vdots&\vdots&\ddots&\vdots\\
-\Kcal_{i_{N-1}i_1}^*&-\Kcal_{i_{N-1}i_2}^*&\cdots&A_{i_{N-1}}
\end{pmatrix}.
\]
This is an $(N-1)\times(N-1)$ operator-block matrix.
For $N=2$, it consists of the single remaining diagonal block.
We write $\Kcal_{j\mathcal J_j}^*=(\Kcal_{ji}^*)_{i\in\mathcal J_j}$
for the $1\times(N-1)$ block row and
$\Kcal_{\mathcal J_jj}^*$ for the $(N-1)\times1$ block column.
Proposition~\ref{prop:material-invertibility} applied to the remaining
inclusions makes $\mathcal A_{\mathcal J_j}$ invertible. Its Schur
complement and the eliminated equation are
\begin{equation}
\mathcal C_j=A_j-\Kcal_{j\mathcal J_j}^*\mathcal A_{\mathcal J_j}^{-1}\Kcal_{\mathcal J_jj}^*,\qquad
\varphi_j=\mathcal C_j^{-1}\!\left(\p_{\nu_j}H+\Kcal_{j\mathcal J_j}^*\mathcal A_{\mathcal J_j}^{-1}(\p_{\nu_i}H)_{i\in\mathcal J_j}\right).
\label{eq:Schur}
\end{equation}
Here, $\mathcal C_j$ acts on the single boundary space for $\Gamma_j$.
Its invertibility follows from invertibility of the full block operator.

\subsection{Lemniscatic and Faber--Walsh Coordinates}
\label{sec:FW}

Walsh's theorem gives a lemniscatic exterior domain
\begin{equation}
    \Lcal
    :=
    \{w\in\widehat{\CC}:|V(w)|>\mu\},
    \qquad
    V(w):=\prod_{j=1}^N(w-b_j)^{m_j},
    \label{eq:LemDomain}
\end{equation}
where $b_1,\ldots,b_N\in\CC$ are pairwise distinct lemniscatic centers,
$\mu>0$ is the lemniscatic level parameter, and the Walsh weights satisfy
$m_j>0$ and $\sum_{j=1}^Nm_j=1$. These fixed weights are distinct
from the integer polynomial indices $m,n$. Fractional powers
may require local branch choices, but
$|V(w)|=\prod_{j=1}^N|w-b_j|^{m_j}$ is single-valued.
The canonical complement has $N$ components, one containing each center.
These weights describe the geometry and are independent of the
conductivities.
Walsh's theorem gives a unique lemniscatic domain and map after
fixing the translational normalization at infinity. One standard choice
for the forward map is $\Phi(z)=z+O(z^{-1})$.  If only the derivative
normalization is imposed, the lemniscatic coordinate is determined up
to a translation of the $w$-plane.  We fix one such
translational representative and denote its inverse by
\begin{equation}
    \Psi:\Lcal\longrightarrow\Omega,
    \qquad
    \Psi(\infty)=\infty,
    \qquad
    \Psi'(\infty)=1.
    \label{eq:PsiWalsh}
\end{equation}
The map extends continuously to the boundary because
$\Gamma_1,\ldots,\Gamma_N$ are Jordan curves.  In a neighborhood of infinity,
\begin{equation}
    \Psi(w)
    =
    w+a_0+\sum_{k=1}^{\infty}a_kw^{-k}.
    \label{eq:PsiLaurent}
\end{equation}
The Laurent expansion \eqref{eq:PsiLaurent} is local at infinity.
Its coefficients depend on the chosen translational representative.

We fix an admissible Walsh sequence
$\beta_1,\beta_2,\ldots\in\{b_1,\ldots,b_N\}$.  We use the standard Walsh choice
for which, for every closed set
$C\subset\widehat{\CC}\setminus\{b_1,\ldots,b_N\}$, there exist constants
$A_1(C),A_2(C)>0$, independent of $n$, such that
\begin{equation}
    A_1(C)|V(w)|^n
    \le |v_n(w)|
    \le A_2(C)|V(w)|^n,
    \qquad w\in C,\quad n\ge0,
    \label{eq:Walsh-un-bound}
\end{equation}
where
\begin{equation}
    v_0(w):=1,
    \qquad
    v_n(w):=\prod_{q=1}^{n}(w-\beta_q).
    \label{eq:un}
\end{equation}
Such sequences exist by Walsh's construction, as stated in
\cite[Lemma~2.2]{SeteLiesen2017}.  In particular, the centers occur with
the prescribed asymptotic frequencies
\[
    \frac{1}{n}
    \#\{1\leq k\leq n:\beta_k=b_j\}
    \longrightarrow m_j,
    \qquad j=1,\ldots,N,
    \quad n\to\infty,
\]
where $\#$ denotes the number of indices in the indicated set.
With $F_0\equiv1$, the FW polynomials $F_n$,
$n\geq1$, are characterized by
\begin{equation}
    \frac{\Psi'(w)}{\Psi(w)-z}
    =
    \sum_{n=0}^{\infty}
    \frac{F_n(z)}{v_{n+1}(w)}.
    \label{eq:FWgen}
\end{equation}
This series is locally uniformly convergent in the standard FW domain of convergence. In particular, for fixed $z$ in a compact subset of $D$, it is locally uniformly convergent for $w$ in the exterior lemniscatic region sufficiently separated from the corresponding level set \cite{SeteLiesen2017}.

Each $F_m$ is monic. We write its coefficients as
\begin{equation}
    F_m(z)
    =
    \sum_{n=0}^{m}p_{mn}z^n,
    \qquad
    p_{mm}=1,
    \label{eq:FWpoly}
\end{equation}
and extend the convention by setting $p_{mn}=0$ for $n>m$. The coefficient matrix
\begin{equation}
    P:=(p_{mn})_{m,n\ge1}
    \label{eq:Pmatrix}
\end{equation}
is unit lower triangular.
Its leading section $P_J=(p_{mn})_{1\le m,n\le J}$ is an invertible
$J\times J$ matrix. In particular,
\begin{equation}
    F_m(z)
    =
    p_{m0}+\sum_{n=1}^{m}p_{mn}z^n,
    \qquad
    \p_\nu F_m
    =
    \sum_{n=1}^{m}p_{mn}\p_\nu z^n.
    \label{eq:Ptransform}
\end{equation}

\begin{remark}[Walsh sequence]
\label{rem:Walsh-sequence}
The FW polynomials and the coefficient matrix $P$ depend on the
chosen Walsh sequence.  This choice changes the polynomial basis but does
not affect the physical solution.
\end{remark}

For $n\ge1$, we define
\begin{equation}
    Q_n(w)
    :=
    \int_w^{\infty}\frac{d\xi}{v_{n+1}(\xi)}.
    \label{eq:Qn}
\end{equation}
The path is taken in a fixed simply connected subdomain of $\Lcal$
containing infinity. All zeros of $v_{n+1}$ lie outside $\Lcal$, and
$v_{n+1}(\xi)^{-1}=O(|\xi|^{-n-1})$ as $|\xi|\to\infty$.
Thus, the integral converges for $n\ge1$, and integrals of its integrand
over large circles tend to zero. Together with simple connectivity,
Cauchy's theorem gives path independence within the chosen domain.
This defines a single-valued branch of $Q_n$ normalized by
$Q_n(\infty)=0$, with
\[
    Q_n'(w)=-\frac{1}{v_{n+1}(w)},
    \qquad
    Q_n(w)=\frac{1}{nw^n}+O(w^{-n-1})
    \quad(w\to\infty).
\]

For a level factor $\eta>1$, we set
\begin{equation}\Lcal_\eta
    :=
    \{w\in\widehat{\CC}:|V(w)|>\eta\mu\},
    \quad
    E_\eta
    :=
    \widehat{\CC}\setminus\Psi(\Lcal_\eta),
    \quad
    \p E_\eta=\Psi(\p\Lcal_\eta).
    \label{eq:Walsh-superlevel}
\end{equation}
For every $\eta>1$, the original compact set
$\overline{D}$ lies strictly inside $E_\eta$.

\begin{proposition}[Logarithmic expansion]
\label{prop:logFW}
Fix regular level factors $1<\eta<\kappa$, and let
$\mathcal O_\kappa\subset\Lcal_\kappa$ be a simply connected subdomain
containing infinity.  Choose the branches of $Q_n$ on
$\mathcal O_\kappa$ by \eqref{eq:Qn}.  Then, for $z$ in a compact
subset of $\operatorname{int}E_\eta$ and
$w\in\mathcal O_\kappa$,
\begin{equation}
    \log\!\left(\frac{\Psi(w)-z}{w-\beta_1}\right)
    =-\sum_{n=1}^{\infty}F_n(z)Q_n(w),
    \label{eq:logFW}
\end{equation}
where the logarithm of the quotient is normalized to zero at infinity.
This quotient is nonzero and holomorphic on $\mathcal O_\kappa$,
including infinity, and hence has a holomorphic logarithm there.
The convergence is
locally uniform in $w$ and uniform for $z$ in such compact subsets.  In
particular, since $\Gamma$ is compact and lies strictly
inside $E_\eta$, the formula is uniform for
$z\in\Gamma$ on compact subsets of
$\mathcal O_\kappa$.
\end{proposition}

\begin{proof}
The generating formula in \cite[Theorem~2.3(1)]{SeteLiesen2017} gives \eqref{eq:FWgen} for
$z\in\p E_\eta$ and $w$ on a strictly larger Walsh level.  We first
record why the same identity is valid uniformly for $z$ in compact
subsets of $\operatorname{int}E_\eta$. We fix a compact set
$W\Subset\mathcal O_\kappa\setminus\{\infty\}$ and choose an
intermediate regular level $\eta_1$ with $\eta<\eta_1<\kappa$.
We estimate the numerator and denominator separately, keeping all
constants independent of $n$.

For the numerator, the integral formula in \cite[Theorem~2.3(1), Eq.~(2.6)]{SeteLiesen2017} gives the
contour representation
\[
    F_n(z)
    =\frac{1}{2\pi \mathrm{i}}
    \int_{\p\Lcal_{\eta_1}}
    \frac{v_n(\tau)\Psi'(\tau)}{\Psi(\tau)-z}\,d\tau,
    \qquad z\in E_\eta,
\]
where all components of $\p\Lcal_{\eta_1}$ are oriented positively
around the bounded complementary regions.  Since $\eta_1>\eta$,
the compact sets $\p E_{\eta_1}=\Psi(\p\Lcal_{\eta_1})$ and
$E_\eta$ are disjoint.  Thus,
\[
    \operatorname{dist}(E_\eta,\p E_{\eta_1})>0.
\]
Moreover, $|V(\tau)|=\eta_1\mu$ on $\p\Lcal_{\eta_1}$, so the upper
bound in \eqref{eq:Walsh-un-bound} implies
\[
    |v_n(\tau)|
    \le A_2(\p\Lcal_{\eta_1})(\eta_1\mu)^n,
    \qquad \tau\in\p\Lcal_{\eta_1}.
\]
Taking absolute values in the contour integral therefore yields
\[
    \sup_{z\in E_\eta}|F_n(z)|
    \le
    \frac{A_2(\p\Lcal_{\eta_1})}
         {2\pi\operatorname{dist}(E_\eta,\p E_{\eta_1})}
    \left(\int_{\p\Lcal_{\eta_1}}|\Psi'(\tau)|\,|d\tau|\right)
    (\eta_1\mu)^n
    =:C_{\eta,\eta_1}(\eta_1\mu)^n.
\]
The constant is finite because the level is regular and $\Psi$ is
analytic in a neighborhood of $\p\Lcal_{\eta_1}$.

For the denominator, $W\subset\Lcal_\kappa$ implies
$|V(w)|>\kappa\mu$ on $W$.  Applying the lower bound in
\eqref{eq:Walsh-un-bound} with $C=W$ and index $n+1$ gives
\[
    |v_{n+1}(w)|
    \ge A_1(W)|V(w)|^{n+1}
    \ge A_1(W)(\kappa\mu)^{n+1},
    \qquad w\in W,
    \qquad A_1(W)>0.
\]
Combining these two estimates, we obtain
\[
    \sup_{\substack{z\in E_\eta\\ w\in W}}
    \left|\frac{F_n(z)}{v_{n+1}(w)}\right|
    \le
    \frac{C_{\eta,\eta_1}(\eta_1\mu)^n}
         {A_1(W)(\kappa\mu)^{n+1}}
    =\frac{C_{\eta,\eta_1}}{A_1(W)\kappa\mu}
    \left(\frac{\eta_1}{\kappa}\right)^n.
\]
Since $\eta_1<\kappa$, the Weierstrass test gives absolute and uniform convergence
on $E_\eta\times W$.  For each fixed $w\in W$, the resulting sum and
$\Psi'(w)/(\Psi(w)-z)$ are analytic in every component of
$\operatorname{int}E_\eta$ and continuous up to its boundary.
Since $w\in\Lcal_\kappa\subset\Lcal_\eta$, the possible pole
$\Psi(w)$ belongs to $\Psi(\Lcal_\eta)=\widehat\CC\setminus E_\eta$.
The regular Green level $\p E_\eta$ encloses finitely many bounded
simply connected components: a bounded component of its exterior would
contradict the maximum principle for the Green function at infinity.
On the closure of each interior component, the difference is continuous,
holomorphic in the interior, and zero on the boundary by
\eqref{eq:FWgen}. The maximum modulus principle therefore extends the
identity throughout $\operatorname{int}E_\eta$.

Subtracting the $n=0$ term from \eqref{eq:FWgen} and using
$F_0=1$ and $v_1(\xi)=\xi-\beta_1$, we obtain
\[
    \frac{\Psi'(\xi)}{\Psi(\xi)-z}
    -\frac{1}{\xi-\beta_1}
    =\sum_{n=1}^{\infty}\frac{F_n(z)}{v_{n+1}(\xi)}.
\]
The difference is $O(|\xi|^{-2})$ at infinity. To integrate the series,
we estimate a bounded path and a radial tail separately.

Since $\mathcal O_\kappa$ contains infinity, we choose $R$ sufficiently
large that $\{|\xi|\ge R\}\subset\mathcal O_\kappa$, that
$R>2\eta_1\mu$, and that
\[
    |V(\xi)|\ge\frac{|\xi|}{2},
    \qquad |\xi|\ge R.
\]
The last inequality follows from $|V(\xi)|/|\xi|\to1$.
For a fixed $w\in\mathcal O_\kappa\setminus\{\infty\}$, we choose the
integration path to consist of the positive real ray from infinity
to $R$, followed by a piecewise smooth path from $R$ to $w$ inside
$\mathcal O_\kappa$.  The latter path is compact and has finite length,
so the preceding uniform convergence on compact sets permits termwise
integration along it, uniformly for $z\in E_\eta$.

On the radial tail, the lower bound in \eqref{eq:Walsh-un-bound}
and the preceding estimate for $F_n$ imply
\[
    \sup_{z\in E_\eta}
    \left|\frac{F_n(z)}{v_{n+1}(\xi)}\right|
    \le
    \frac{C_{\eta,\eta_1}(\eta_1\mu)^n}
         {A_1(\{|\xi|\ge R\})(|\xi|/2)^{n+1}}
    =
    \frac{2C_{\eta,\eta_1}}
         {A_1(\{|\xi|\ge R\})}
    \frac{(2\eta_1\mu)^n}{|\xi|^{n+1}},
    \qquad n\ge1.
\]
Here, the Walsh constant is independent of $n$.  Integrating this
bound with $\xi=r\ge R$ gives
\[
    \int_R^\infty
    \sup_{z\in E_\eta}
    \left|\frac{F_n(z)}{v_{n+1}(r)}\right|\,dr
    \le
    \frac{2C_{\eta,\eta_1}}
         {A_1(\{|\xi|\ge R\})}
    \frac{1}{n}\left(\frac{2\eta_1\mu}{R}\right)^n,
\]
since $\int_R^\infty r^{-n-1}\,dr=R^{-n}/n$.
Because $2\eta_1\mu/R<1$, these bounds are summable over $n\ge1$.
Thus, the sum of the absolute integrals is finite, uniformly in
$z\in E_\eta$, which justifies exchanging the sum and the improper
integral on the tail.  Together with the bounded-path argument, this
justifies termwise integration along the complete path.  For $w$ in
a compact subset of $\mathcal O_\kappa\setminus\{\infty\}$, finitely
many fixed paths followed by short local segments give the same
argument uniformly in $w$.  Thus,
\[
    \log\!\left(\frac{\Psi(w)-z}{w-\beta_1}\right) = \int_{\infty}^{w} \left( \frac{\Psi'(\xi)}{\Psi(\xi)-z} -\frac{1}{\xi-\beta_1} \right)d\xi = \sum_{n=1}^{\infty}F_n(z) \int_{\infty}^{w} \frac{d\xi}{v_{n+1}(\xi)} =- \sum_{n=1}^{\infty}F_n(z)Q_n(w).
\]
The logarithm of the quotient is obtained by continuation from infinity,
where $(\Psi(w)-z)/(w-\beta_1)\to1$.  This proves
\eqref{eq:logFW} and its stated uniformity.
\end{proof}

Here, $\RePart$ and $\ImPart$ denote the real and imaginary parts,
respectively. Taking real parts, we obtain
\begin{equation}
    \log|\Psi(w)-z|
    =
    \log|w-\beta_1|
    -\RePart\!\left[\sum_{n=1}^{\infty}F_n(z)Q_n(w)\right].
    \label{eq:realLogFW}
\end{equation}
We now consider
$\boldsymbol\varphi=(\varphi_j)_{j=1}^N\in
\prod_{j=1}^N L^1(\Gamma_j)$ that is real-valued and has zero charge on
each component.
By the uniform convergence in Proposition~\ref{prop:logFW}, the series
in \eqref{eq:realLogFW} may be integrated term by term over the boundaries on compact subsets of a fixed exterior branch domain.  The
term $\log|w-\beta_1|$ is independent of the boundary variable and
vanishes after integration because of the componentwise zero-charge
condition.  Therefore,
\begin{equation}
    \sum_{j=1}^{N}\Scal_{\Gamma_j}[\varphi_j](\Psi(w))
    =-\frac{1}{2\pi}\RePart\!\left[
    \sum_{n=1}^{\infty}Q_n(w)
    \sum_{j=1}^{N}\int_{\Gamma_j}F_n(z)\varphi_j(z)\,ds(z)
    \right].
    \label{eq:SLFW}
\end{equation}

\paragraph{Branch dependence of the outgoing modes.}

The functions $Q_n$ are defined on a fixed exterior branch domain and
need not be single-valued on the whole lemniscatic exterior. However,
for a zero-charge density $\boldsymbol\varphi$, the complete combination
\[
\mathscr P_{\boldsymbol\varphi}(w)
:=
\sum_{n\ge1}
\left(
\sum_{j=1}^N
\int_{\Gamma_j}F_n(z)\varphi_j(z)\,ds(z)
\right)Q_n(w)
\]
has a single-valued holomorphic continuation to $\Lcal$, normalized to
vanish at infinity. Indeed,
\[
\mathscr P_{\boldsymbol\varphi}'(w)
=
-\Psi'(w)\sum_{j=1}^N
\int_{\Gamma_j}
\frac{\varphi_j(z)}{\Psi(w)-z}\,ds(z),
\]
and its periods vanish by the componentwise zero-charge condition.

Thus, although the individual $Q_n$ and finite outgoing sums may depend
on branch choices, the complete potential in \eqref{eq:SLFW} is
single-valued.


\paragraph{Rotationally symmetric $N$-component geometry.}
For $N\ge2$, a useful canonical family is obtained by setting
\[
    b_j=b\exp\!\left(\frac{2\pi \mathrm{i}(j-1)}{N}\right),\qquad
    m_j=\frac1N,\qquad j=1,\ldots,N,\qquad 0<\mu<b.
\]
Then
\begin{equation}
    \mathcal L_{N;b,\mu}
    :=\{w\in\widehat\CC:|w^N-b^N|>\mu^N\}.
    \label{eq:N-symmetric-lemniscate}
\end{equation}
The $N$ bounded complementary components are disjoint. They admit the
counterclockwise boundary parameterizations
\begin{equation}
    w_j(t)=b_j\left(1+(\mu/b)^N e^{\mathrm{i}t}\right)^{1/N},
    \qquad 0\le t<2\pi,
    \label{eq:N-symmetric-param}
\end{equation}
where the root is the branch taking the value $1$ at the center of
the disk $\{|\zeta-1|<(\mu/b)^N\}$. Each parameterization extends
conformally to $|e^{\mathrm{i}t}|<1$. Its angular origin need not give the
positive-real derivative normalization used later for $\chi_j$.
For the cyclic Walsh sequence $\beta_{nN+j}=b_j$, $n\ge0$,
$1\le j\le N$, one has
\begin{equation}
    v_{nN+m}(w)=(w^N-b^N)^n\prod_{i=1}^{m}(w-b_i),
    \qquad n\ge0,\quad 0\le m<N.
    \label{eq:N-symmetric-Walsh}
\end{equation}
The empty product is $1$. Each cycle uses every center once, giving
bounded discrepancy and hence \eqref{eq:Walsh-un-bound}. This canonical
family is used in the first numerical example in Section~\ref{sec:numerics}.

\section{FWPTs and Grunsky Factorization}
\label{sec:tensors}

\subsection{FWPTs and Multipole Expansion}

The CGPTs record the response to incident harmonic modes. We first
define them using complex monomials and then introduce their
FW counterparts, the FWPTs.

\begin{definition}[CGPTs]
\label{def:CGPTs}
For $m\in\NN$, we set $Z_m(z):=z^m$ and write
$\boldsymbol g_f=(\p_{\nu_j}f)_{j=1}^N$ for the boundary normal data
of a harmonic polynomial $f$. For $m,n\in\NN$, we define
\begin{align}
    \mathbb N_{mn}^{(1)}(D,\boldsymbol\lambda)
    &:=
    \sum_{j=1}^{N}
    \int_{\Gamma_j}
    Z_n(z)
    \left[
    (\Lambda_\Gamma-\mathcal K_\Gamma^*)^{-1}
    \boldsymbol g_{Z_m}
    \right]_j(z)\,ds(z),
    \label{eq:CGPT1}\\
    \mathbb N_{mn}^{(2)}(D,\boldsymbol\lambda)
    &:=
    \sum_{j=1}^{N}
    \int_{\Gamma_j}
    Z_n(z)
    \left[
    (\Lambda_\Gamma-\mathcal K_\Gamma^*)^{-1}
    \boldsymbol g_{\overline{Z_m}}
    \right]_j(z)\,ds(z),
    \label{eq:CGPT2}
\end{align}
and set
\begin{equation}
    \mathbb N^{(1)}=(\mathbb N_{mn}^{(1)})_{m,n\ge1},
    \qquad
    \mathbb N^{(2)}=(\mathbb N_{mn}^{(2)})_{m,n\ge1}.
    \label{eq:CGPTmatrices}
\end{equation}
Here, $[\cdot]_j$ denotes the component on $\Gamma_j$, and an overline denotes complex conjugation.
\end{definition}

Expanding the incident and testing monomials in Cartesian coordinates
shows that these CGPTs are complex linear combinations of the GPTs
\cite{AmmariKang2007,ChoiKimLim2023}.

Let $H$ be a real-valued entire harmonic incident field with expansion
\begin{equation}
    H(z)
    =
    \gamma_0+
    \sum_{m=1}^{\infty}
    \left(
    \gamma_m Z_m(z)+
    \overline{\gamma_m}\,\overline{Z_m(z)}
    \right),
    \label{eq:Hmono}
\end{equation}
where $\gamma_0\in\RR$ and $\gamma_m\in\CC$. Its classical CGPT expansion is
\begin{equation}
u(z)-H(z)
=
-\frac{1}{2\pi}\RePart\!\left[
\sum_{n=1}^{\infty}
\sum_{m=1}^{\infty}
\left(
\gamma_m\mathbb N_{mn}^{(1)}
+\overline{\gamma_m}\mathbb N_{mn}^{(2)}
\right)
\frac{z^{-n}}{n}
\right].
\label{eq:CGPTfar}
\end{equation}
This expansion holds for $|z|$ outside a disk containing $D$, as in
\cite{AmmariKang2004,ChoiHelsingKangLim2023,ChoiKimLim2023}.

Replacing the monomials in the CGPTs by FW polynomials
gives the FWPTs, which include moments of all positive orders.

\begin{definition}[FWPTs]
\label{def:FWPTs}
Let $F_m$ denote the $m$th FW polynomial associated with the
fixed admissible Walsh sequence, as defined in \eqref{eq:FWgen}.
For $m,n\in\NN$, we define
\begin{align}
    \mathbb F_{mn}^{(1)}(D,\boldsymbol\lambda)
    &:=
    \sum_{j=1}^{N}
    \int_{\Gamma_j}
    F_n(z)
    \left[
    (\Lambda_\Gamma-\mathcal K_\Gamma^*)^{-1}
    \boldsymbol g_{F_m}
    \right]_j(z)\,ds(z),
    \label{eq:FWPT1}\\
    \mathbb F_{mn}^{(2)}(D,\boldsymbol\lambda)
    &:=
    \sum_{j=1}^{N}
    \int_{\Gamma_j}
    F_n(z)
    \left[
    (\Lambda_\Gamma-\mathcal K_\Gamma^*)^{-1}
    \boldsymbol g_{\overline{F_m}}
    \right]_j(z)\,ds(z),
    \label{eq:FWPT2}
\end{align}
and set
\begin{equation}
    \mathbb F^{(1)}=(\mathbb F_{mn}^{(1)})_{m,n\ge1},
    \qquad
    \mathbb F^{(2)}=(\mathbb F_{mn}^{(2)})_{m,n\ge1}.
    \label{eq:FWPTmatrices}
\end{equation}
\end{definition}

We assume that the real-valued entire harmonic incident field $H$ admits the
following expansion on $\overline D$, with $\alpha_0\in\RR$ and
$\alpha_m\in\CC$:
\begin{equation}
    H(z)
    =
    \alpha_0+
    \sum_{m=1}^{\infty}
    \left(
    \alpha_mF_m(z)+
    \overline{\alpha_m}\,\overline{F_m(z)}
    \right).
    \label{eq:HFW}
\end{equation}
We denote by $H_d$ the partial sum through $m=d$. For an infinite
expansion, the convergence required here is
\[
    \|H_d-H\|_{C(\overline D)}
    +\left(\sum_{j=1}^N
    \|\p_{\nu_j}(H_d-H)\|_{L^2(\Gamma_j)}^2\right)^{1/2}
    \longrightarrow0,
    \qquad d\longrightarrow\infty,
\]
where $C(\overline D)$ carries the uniform norm. The normal data have
zero mean on every component by Lemma~\ref{lem:zerocharge}.
For harmonic polynomials, the expansion is finite, and these conditions
are automatic. We fix regular levels $1<\eta<\kappa$ and a branch domain
$\mathcal O_\kappa$ as in Proposition~\ref{prop:logFW}.
Then Lemma~\ref{lem:zerocharge} and \eqref{eq:SLFW} give
\begin{equation}
\begin{split}
u(\Psi(w))-H(\Psi(w))
=
-\frac{1}{2\pi}\RePart\!\left[
\sum_{n=1}^{\infty}
\sum_{m=1}^{\infty}
\left(
\alpha_m\mathbb F_{mn}^{(1)}
+\overline{\alpha_m}\mathbb F_{mn}^{(2)}
\right)
Q_n(w)
\right].
\end{split}
\label{eq:GME}
\end{equation}
The braces denote the limit of the incident partial sums in their
prescribed order for each fixed $n$. We take that limit before the
outgoing sum. To justify it, Proposition~\ref{prop:material-invertibility}
gives
$\boldsymbol\varphi_d
=(\Lambda_\Gamma-\mathcal K_\Gamma^*)^{-1}\boldsymbol g_{H_d}
\to\boldsymbol\varphi$ in $\prod_{j=1}^N L_0^2(\Gamma_j)$.
By linearity, the incident partial sum in braces is
$\mathfrak m_n(\boldsymbol\varphi_d)$. For each fixed $n$,
\[
\begin{split}
    |\mathfrak m_n(\boldsymbol\varphi_d)
      -\mathfrak m_n(\boldsymbol\varphi)|
    &\le
    \left(\sum_{j=1}^N\|F_n\|_{L^2(\Gamma_j)}^2\right)^{1/2}
    \left(\sum_{j=1}^N
    \|\varphi_{d,j}-\varphi_j\|_{L^2(\Gamma_j)}^2\right)^{1/2}
    \longrightarrow0.
\end{split}
\]
Moreover, $\sup_d\sum_j\|\varphi_{d,j}\|_{L^1(\Gamma_j)}<\infty$.
The bounds \eqref{eq:Fn-boundary-bound} and
\eqref{eq:Qn-geometric-bound}, proved below, therefore give
\[
    \sup_{w\in K}
    |Q_n(w)\mathfrak m_n(\boldsymbol\varphi_d)|
    \le C_K(\eta/\kappa)^n
\]
for $K\Subset\mathcal O_\kappa\setminus\{\infty\}$, with $C_K$
independent of $d,n$. Dominated convergence in the outgoing index now
gives \eqref{eq:GME} locally uniformly on the fixed branch domain.
We do not rearrange the incident series or assert absolute convergence
in its index.
The physical domain of this representation is $\Psi(\mathcal O_\kappa)$,
which contains a neighborhood of infinity because $\Psi(w)=w+O(1)$.
Thus, \eqref{eq:GME} and the Cartesian expansion \eqref{eq:CGPTfar}
represent the same field outside a sufficiently large disk, although
Walsh level curves need not be circles.

\paragraph{Convergence of the truncated multipole expansion.}
We estimate the error obtained by retaining finitely many outgoing modes
in \eqref{eq:GME}, with the exact expansion coefficients.  The level sets
$\Lcal_\eta$ and $E_\eta$ are defined in \eqref{eq:Walsh-superlevel}.

\begin{theorem}[Geometric convergence of the FW truncation]
\label{thm:FW-truncation}
Let $H$ be a real-valued harmonic polynomial and let
$\boldsymbol\varphi$ be the solution of \eqref{eq:blockres}.
Since the monic polynomials $F_0=1,F_1,\ldots,F_{d}$ form a basis for
polynomials of degree at most $d$, we may write
\[
    H(z)=\alpha_0+
    \sum_{m=1}^{d}
    \left(\alpha_mF_m(z)
    +\overline{\alpha_m}\,\overline{F_m(z)}\right),
    \qquad \alpha_0\in\RR,\quad \alpha_m\in\CC.
\]
Here, $d\ge1$ can be chosen even when $H$ is constant, by setting all
$\alpha_m=0$.  The boundary integrals satisfy
\begin{equation}
    \sum_{j=1}^{N}\int_{\Gamma_j}F_n(z)\varphi_j(z)\,ds(z)
    =\sum_{m=1}^{d}
    \left(\alpha_m\mathbb F_{mn}^{(1)}(D,\boldsymbol\lambda)
    +\overline{\alpha_m}\mathbb F_{mn}^{(2)}(D,\boldsymbol\lambda)\right),
    \qquad n\ge1.
    \label{eq:Cn-general}
\end{equation}
Fix $1<\eta<\kappa$, with $\eta$ a regular Walsh level, and a simply
connected subdomain $\mathcal O_\kappa$ of $\Lcal_\kappa$ containing
infinity.  Let $K\Subset\mathcal O_\kappa\setminus\{\infty\}$ be compact,
and use the branches of $Q_n$ on $\mathcal O_\kappa$ normalized by
\eqref{eq:Qn}.  For $J\ge0$, define
\begin{equation}
u^{(J)}(\Psi(w))
:=
H(\Psi(w))
-
\frac{1}{2\pi}\RePart\!\left[
\sum_{n=1}^{J}
\sum_{m=1}^{d}
\left(
\alpha_m\mathbb F_{mn}^{(1)}(D,\boldsymbol\lambda)
+
\overline{\alpha_m}\mathbb F_{mn}^{(2)}(D,\boldsymbol\lambda)
\right)
Q_n(w)
\right].
\label{eq:FW-truncated-general}
\end{equation}
where the sum is empty for $J=0$.
There exists a constant $C_{K,\eta,\kappa,H}>0$, independent of $J$,
such that
\begin{equation}
    \sup_{w\in K}|u(\Psi(w))-u^{(J)}(\Psi(w))|
    \le C_{K,\eta,\kappa,H}
    \left(\frac{\eta}{\kappa}\right)^{J+1},
    \qquad J\ge0.
    \label{eq:FW-truncation-estimate}
\end{equation}
The constant may also depend on the fixed geometry, contrast vector $\boldsymbol\lambda$,
Walsh sequence, and branch domain $\mathcal O_\kappa$.
Thus, the truncations converge uniformly on $K$ at a geometric rate.
\end{theorem}

\begin{proof}
We estimate the moments and the outgoing functions separately.

The boundaries lie strictly inside every $E_\eta$ with
$\eta>1$. The integral formula in \cite[Theorem~2.3(1), Eq.~(2.6)]{SeteLiesen2017}, applied with an
intermediate level between $1$ and $\eta$, gives
\begin{equation}
    F_n(z)=\frac{1}{2\pi \mathrm{i}}
    \int_{\p\Lcal_\eta}
    v_n(\tau)\frac{\Psi'(\tau)}{\Psi(\tau)-z}\,d\tau,
    \qquad z\in\Gamma.
    \label{eq:Fn-contour-representation}
\end{equation}
All contour components are oriented positively around the bounded
complementary regions.  The regularity of the level ensures finite
contour length, and
$\operatorname{dist}(\Gamma,\p E_\eta)>0$.
Since $|V(\tau)|=\eta\mu$ on $\p\Lcal_\eta$, the upper bound in
\eqref{eq:Walsh-un-bound} yields
\begin{equation}
    \max_{z\in\Gamma}|F_n(z)|
    \le C_\eta(\eta\mu)^n,
    \qquad
    C_\eta:=
    \frac{A_2(\p\Lcal_\eta)
    \displaystyle\int_{\p\Lcal_\eta}|\Psi'(\tau)|\,|d\tau|}
    {2\pi\operatorname{dist}(\Gamma,\p E_\eta)}.
    \label{eq:Fn-boundary-bound}
\end{equation}
Since $\varphi_j\in L^1(\Gamma_j)$, as noted after
Lemma~\ref{lem:zerocharge}, estimate~\eqref{eq:Fn-boundary-bound} gives
\begin{equation}
    \left|\sum_{j=1}^{N}\int_{\Gamma_j}F_n(z)\varphi_j(z)\,ds(z)\right| \le\sum_{j=1}^{N}\int_{\Gamma_j} |F_n(z)|\,|\varphi_j(z)|\,ds(z) \le C_\eta \left(\sum_{j=1}^{N}\|\varphi_j\|_{L^1(\Gamma_j)}\right) (\eta\mu)^n. \label{eq:Cn-growth}
\end{equation}

We next bound $Q_n$ uniformly on $K$. The path independence established
at \eqref{eq:Qn} allows us to choose convenient integration paths within
$\mathcal O_\kappa$. First, we choose $R$ large enough that
\begin{equation}
    \{|\xi|\ge R\}\subset\mathcal O_\kappa,
    \qquad |V(\xi)|\ge\frac{|\xi|}{2}\quad(|\xi|\ge R),
    \qquad R>2\kappa\mu.
    \label{eq:U-tail-lower}
\end{equation}
Such an $R$ exists because $\mathcal O_\kappa$ contains infinity and
\[
    \frac{|V(\xi)|}{|\xi|}
    =\prod_{j=1}^{N}\left|1-\frac{b_j}{\xi}\right|^{m_j}
    \longrightarrow1.
\]
We integrate from $w$ to the positive real point $R$ along a bounded
path, and then from $R$ to infinity along the positive real ray.

The bounded paths can be chosen to lie in a compact set
$C_K\Subset\mathcal O_\kappa\setminus\{\infty\}$ and to have lengths
bounded by a constant $L_K$. Indeed, we cover $K$ by finitely many disks
whose closures lie in the finite part of $\mathcal O_\kappa$ and join
each center to $R$ by a fixed piecewise smooth path.  A point in each
disk is joined to its center by a segment and then follows the fixed
path to $R$.  The finite union of these closed disks and paths provides
$C_K$, and their finite lengths give $L_K$.
Applying \eqref{eq:Walsh-un-bound} to $C_K$, where
$|V(\xi)|>\kappa\mu$, gives
\begin{equation}
    \left|\int_w^R\frac{d\xi}{v_{n+1}(\xi)}\right|
    \le\frac{L_K}{A_1(C_K)\kappa\mu}(\kappa\mu)^{-n}.
    \label{eq:Qn-bounded-path}
\end{equation}

On the radial tail, we apply the same Walsh lower bound to the closed
exterior set $\{|\xi|\ge R\}$.  By \eqref{eq:U-tail-lower},
\begin{align}
    &\left|\int_R^\infty\frac{d\xi}{v_{n+1}(\xi)}\right| \le\frac{1}{A_1(\{|\xi|\ge R\})} \int_R^\infty\left(\frac{2}{r}\right)^{n+1}\,dr =\frac{2}{nA_1(\{|\xi|\ge R\})} \left(\frac{R}{2}\right)^{-n}\nonumber\\
    &\le\frac{2}{A_1(\{|\xi|\ge R\})}(\kappa\mu)^{-n},
    \qquad n\ge1. \label{eq:Qn-radial-tail}
\end{align}
Combining \eqref{eq:Qn-bounded-path} and \eqref{eq:Qn-radial-tail},
we obtain
\begin{equation}
    \sup_{w\in K}|Q_n(w)|\le B_{K,\kappa}(\kappa\mu)^{-n},
    \qquad n\ge1,
    \label{eq:Qn-geometric-bound}
\end{equation}
where we may take
\[
    B_{K,\kappa}=
    \frac{L_K}{A_1(C_K)\kappa\mu}
    +\frac{2}{A_1(\{|\xi|\ge R\})}.
\]
The choices of $R$, $C_K$, and $L_K$ were made independently of $n$,
so $B_{K,\kappa}$ is also independent of $n$.

Together with \eqref{eq:Cn-growth}, this gives
\begin{equation}
    \sup_{w\in K}\left|Q_n(w)\sum_{j=1}^{N}\int_{\Gamma_j}F_n(z)\varphi_j(z)\,ds(z)\right|
    \le C_\eta B_{K,\kappa}
    \left(\sum_{j=1}^{N}\|\varphi_j\|_{L^1(\Gamma_j)}\right)
    \left(\frac{\eta}{\kappa}\right)^n.
    \label{eq:CnQn-bound}
\end{equation}
In particular, the series is absolutely and uniformly convergent on $K$.
Since $H$ and all $\lambda_j$ are real and the block operator has a real
kernel, uniqueness implies that $\boldsymbol\varphi$ is real-valued.
Its componentwise zero charge, together with \eqref{eq:SLansatz} and
\eqref{eq:SLFW}, therefore gives
\begin{equation}
    u(\Psi(w))-H(\Psi(w))
    =-\frac{1}{2\pi}\RePart\!\left[
    \sum_{n=1}^{\infty}Q_n(w)\sum_{m=1}^{d}\left(\alpha_m\mathbb F_{mn}^{(1)}(D,\boldsymbol\lambda)+\overline{\alpha_m}\mathbb F_{mn}^{(2)}(D,\boldsymbol\lambda)\right)\right].
    \label{eq:GME-Cn}
\end{equation}
Here, we used \eqref{eq:Cn-general} to express the boundary integrals
in terms of the FWPTs, as in \eqref{eq:GME}.
Subtracting the truncation and summing the geometric tail, we obtain
\begin{align*}
    \sup_{w\in K}|u(\Psi(w))-u^{(J)}(\Psi(w))|
    &\le\frac{1}{2\pi}\sum_{n=J+1}^{\infty}
    \sup_{w\in K}\left|Q_n(w)\sum_{j=1}^{N}\int_{\Gamma_j}F_n(z)\varphi_j(z)\,ds(z)\right|\\
    &\le
    \frac{C_\eta B_{K,\kappa}}
    {2\pi(1-\eta/\kappa)}
    \left(\sum_{j=1}^{N}\|\varphi_j\|_{L^1(\Gamma_j)}\right)
    \left(\frac{\eta}{\kappa}\right)^{J+1}.
\end{align*}
This proves \eqref{eq:FW-truncation-estimate}.
\end{proof}

\begin{theorem}[CGPT--FWPT relation]
\label{thm:CGPTFWPT}
With $A^T$ denoting the transpose of a matrix $A$, the CGPT and FWPT
matrices satisfy
\begin{equation}
    \mathbb F^{(1)}=P\mathbb N^{(1)}P^T,
    \qquad
    \mathbb F^{(2)}=\overline P\,\mathbb N^{(2)}P^T.
    \label{eq:CGPTFWPT}
\end{equation}
\end{theorem}

\begin{proof}
From \eqref{eq:Ptransform}, we obtain
\[
    \p_\nu F_m
    =\sum_{p=1}^{m}p_{mp}\p_\nu Z_p,
    \qquad
    \p_\nu\overline{F_m}
    =\sum_{p=1}^{m}\overline{p_{mp}}\p_\nu\overline{Z_p}.
\]
By Lemma~\ref{lem:zerocharge}, the constant term of the testing polynomial
does not contribute.  Bilinearity gives
\[
    \mathbb F_{mn}^{(1)} ={} \sum_{p=1}^{m}\sum_{q=1}^{n} p_{mp}p_{nq}\mathbb N_{pq}^{(1)},\qquad \mathbb F_{mn}^{(2)} ={} \sum_{p=1}^{m}\sum_{q=1}^{n} \overline{p_{mp}}p_{nq}\mathbb N_{pq}^{(2)}.
\]
\end{proof}

\subsection{Walsh--Grunsky Boundary Factorization}
\label{sec:Grunsky}

Throughout this subsection, the interfaces $\Gamma_j$ are analytic
Jordan curves. This assumption permits boundary evaluation of the
reciprocal Walsh and Walsh--Grunsky series. The preceding results
require only $C^{1,\alpha}$ regularity.

In the exterior lemniscatic domain, we use S\`ete's reciprocal
FW expansion \cite[Proposition~3.1(4)]{Sete2013}, which we
derive directly in the proof of Lemma~\ref{lem:analytic-boundary-traces}:
\begin{equation}
    F_m(\Psi(w))
    =
    v_m(w)
    +
    \sum_{n=1}^{\infty}
    \frac{g_{mn}}{v_n(w)},
    \qquad m\ge1.
    \label{eq:FW-Grunsky-expansion}
\end{equation}
We collect these coefficients into
\begin{equation}
    G:=(g_{mn})_{m,n\ge1}
    \label{eq:Walsh-Grunsky-matrix}
\end{equation}
and call $G$ the \emph{Walsh--Grunsky coefficient matrix}. The section
$G_J=(g_{mn})_{1\le m,n\le J}$ has size $J\times J$.
The first index $m$ labels the polynomial $F_m$, and the
second index $n$ labels the reciprocal mode $1/v_n$. No symmetry of
$G$ is assumed.

The following elementary interpolation identity will be useful. For
$w,\tau\in\Lcal\setminus\{\infty\}$ with $|V(\tau)|<|V(w)|$,
\begin{equation}
    \frac{1}{w-\tau}
    =
    \sum_{m=0}^{\infty}
    \frac{v_m(\tau)}{v_{m+1}(w)}.
    \label{eq:Walsh-Newton-kernel}
\end{equation}
Indeed, the recurrence $v_{m+1}(z)=(z-\beta_{m+1})v_m(z)$ gives
\[
    (w-\tau)\frac{v_m(\tau)}{v_{m+1}(w)}
    =\frac{v_m(\tau)}{v_m(w)}
    -\frac{v_{m+1}(\tau)}{v_{m+1}(w)}.
\]
Summing from $m=0$ to $J$ and using $v_0=1$, we obtain
\[
    \sum_{m=0}^{J}\frac{v_m(\tau)}{v_{m+1}(w)}
    =\frac{1}{w-\tau}
    \left(1-\frac{v_{J+1}(\tau)}{v_{J+1}(w)}\right).
\]
By \eqref{eq:Walsh-un-bound}, the remaining quotient satisfies
\[
    \left|\frac{v_{J+1}(\tau)}{v_{J+1}(w)}\right|
    \le\frac{A_2(\{\tau\})}{A_1(\{w\})}
    \left(\frac{|V(\tau)|}{|V(w)|}\right)^{J+1}
    \longrightarrow0.
\]
This proves \eqref{eq:Walsh-Newton-kernel}.

\begin{proposition}[Walsh--Grunsky kernel identity]
\label{prop:WG-kernel}
For $w$ and $\tau$ in separated exterior Walsh level regions where
the series below converge absolutely and locally uniformly,
\begin{equation}
    \frac{\Psi'(w)}
         {\Psi(w)-\Psi(\tau)}
    -
    \frac{1}{w-\tau}
    =
    \sum_{m,n\ge1}
    \frac{g_{mn}}
         {v_{m+1}(w)v_n(\tau)}.
    \label{eq:WG-kernel}
\end{equation}
The singularity at $w=\tau$ is entirely contained in the canonical
kernel $(w-\tau)^{-1}$. The remaining term is the Grunsky correction.
\end{proposition}

\begin{proof}
We set $z=\Psi(\tau)$ in the generating relation \eqref{eq:FWgen} and use
\eqref{eq:FW-Grunsky-expansion}.  The contribution of the canonical
parts $v_m(\tau)$ is exactly \eqref{eq:Walsh-Newton-kernel}, while the
remaining terms give the double series on the right-hand side of
\eqref{eq:WG-kernel}.
\end{proof}

We next prove absolute and uniform convergence near the boundary.
We do this by showing that the sums of the supremum norms of the
summands are finite on compact subsets.
For $\rho>0$, we write
\begin{equation}
    \Lcal^{(\rho)}
    :=
    \{w\in\widehat{\CC}:|V(w)|>\rho\mu\},
    \qquad
    \p\Lcal^{(\rho)}
    =
    \{w\in\CC:|V(w)|=\rho\mu\}.
    \label{eq:inner-Walsh-levels}
\end{equation}
Thus, $\Lcal^{(1)}=\Lcal$.

\begin{lemma}[Analytic boundary continuation and convergence]
\label{lem:analytic-boundary-traces}
Assume that all $\Gamma_j$, $j=1,\ldots,N$, are analytic.  Then there
exists $\rho_0\in(0,1)$ such that the inverse Walsh map $\Psi$ extends to
a one-to-one holomorphic map on $\Lcal^{(\rho_0)}$, with nonvanishing
derivative there.

For the coefficients $g_{mn}$ defined in
\eqref{eq:FW-Grunsky-expansion} and every choice
\begin{equation}
    \rho_0<\rho_1<\rho_2<\rho_3<1,
    \label{eq:collar-levels}
\end{equation}
there is a constant $C>0$, independent of $m,n$, such that
\begin{equation}
    |g_{mn}|
    \le
    C\,(\rho_1\mu)^m(\rho_2\mu)^{n-1},
    \qquad m,n\ge1.
    \label{eq:Grunsky-decay}
\end{equation}
Consequently, for every fixed $m\ge1$, the reciprocal Walsh series
\begin{equation}
    \sum_{n=1}^{\infty}\frac{g_{mn}}{v_n(w)}
    \label{eq:reciprocal-normal-series}
\end{equation}
converges absolutely and uniformly on
$\{w:|V(w)|\ge \rho_3\mu\}$, and the double series
\begin{equation}
    \sum_{m,n\ge1}
    \frac{g_{mn}}
         {v_{m+1}(w)v_n(\tau)}
    \label{eq:kernel-normal-series}
\end{equation}
converges absolutely and uniformly on $\{(w,\tau):|V(w)|\ge \rho_3\mu,\ |V(\tau)|\ge \rho_3\mu\}$.
In particular, the convergence is absolute and uniform in neighborhoods
of $\p\Lcal$ and $\p\Lcal\times\p\Lcal$, respectively. Hence,
\eqref{eq:FW-Grunsky-expansion} and \eqref{eq:WG-kernel} hold on
the canonical boundary, with termwise integration and differentiation
justified.
\end{lemma}

\begin{proof}
\emph{Conformal continuation.}
Since the source and target boundaries are analytic Jordan curves,
Schwarz reflection extends $\Psi$ and its inverse holomorphically across
each boundary component \cite{Pommerenke1992}. These extensions remain
local inverses and hence have nonvanishing derivatives. Since there are
finitely many boundary components, they patch to a holomorphic,
locally one-to-one extension on a neighborhood of $\overline\Lcal$ in
the Riemann sphere, using the local coordinates $1/w$ and $1/\Psi(w)$
at infinity.

After shrinking this neighborhood if necessary, the extension is
one-to-one. Otherwise, there would be distinct points
$w_r,\tau_r$ approaching $\overline\Lcal$ with
$\Psi(w_r)=\Psi(\tau_r)$. By compactness and injectivity on
$\overline\Lcal$, both sequences have the same limit, contradicting
local injectivity near that point. Hence, for some $\rho_0<1$,
$\Psi$ is one-to-one and holomorphic on $\Lcal^{(\rho_0)}$ with
nonvanishing derivative.

\emph{Coefficient estimates.}
On the enlarged domain, the lemniscatic level parameter is $\rho_0\mu$.
The polynomials $F_m$ are unchanged: they are the polynomial parts of
$v_m(\Phi(z))$ at infinity, and neither $\Phi=\Psi^{-1}$ near infinity
nor the Walsh sequence has changed, as explained in
\cite[Corollary~3.1]{SeteLiesen2017}.
We record a direct derivation of the reciprocal coefficient formula
of \cite[Proposition~3.1(4)]{Sete2013}.
We set $h_m(\tau):=F_m(\Psi(\tau))-v_m(\tau)$, which is holomorphic
on $\Lcal^{(\rho_0)}$ and is $O(\tau^{-1})$ at infinity.
Deforming the contour in the defining integral for $F_m$ inward
across the simple pole at $s=\tau$, whose residue is $v_m(\tau)$,
we obtain
\[
    h_m(\tau)=\frac{1}{2\pi\mathrm{i}}
    \int_{\p\Lcal^{(\rho_1)}}
    v_m(s)\frac{\Psi'(s)}{\Psi(s)-\Psi(\tau)}\,ds,
    \qquad |V(\tau)|>\rho_1\mu.
\]
For $|V(w)|>\rho_2\mu$, exterior Cauchy's formula and the telescoping
identity \eqref{eq:Walsh-Newton-kernel} give
\[
\begin{split}
    h_m(w)
    &=\frac{1}{2\pi\mathrm{i}}
    \int_{\p\Lcal^{(\rho_2)}}\frac{h_m(\tau)}{w-\tau}\,d\tau\\
    &=\sum_{n\ge1}\frac{1}{v_n(w)}
    \left[\frac{1}{2\pi\mathrm{i}}
    \int_{\p\Lcal^{(\rho_2)}}v_{n-1}(\tau)h_m(\tau)\,d\tau\right].
\end{split}
\]
The bound \eqref{eq:Walsh-un-bound} justifies termwise integration
on separated Walsh levels. The telescoping proof also applies on the
enlarged domain, since the centers remain outside it.
Uniqueness of the reciprocal expansion identifies the coefficients
in brackets with $g_{mn}$. Substituting the contour formula for $h_m$
therefore yields
\begin{equation}
    g_{mn}=-\frac{1}{4\pi^2}
    \int_{\p\Lcal^{(\rho_2)}}\!
    \int_{\p\Lcal^{(\rho_1)}}
    v_{n-1}(\tau)v_m(s)
    \frac{\Psi'(s)}{\Psi(s)-\Psi(\tau)}\,ds\,d\tau,
    \label{eq:g-contour-collar}
\end{equation}
where the contour components are positively oriented around their
bounded complementary regions. The two image contours are disjoint
compact sets, so the denominator is bounded away from zero. Their
lengths and $|\Psi'|$ on the contours are finite. Moreover,
\eqref{eq:Walsh-un-bound} gives
\[
    |v_m(s)|\le A_2(\p\Lcal^{(\rho_1)})(\rho_1\mu)^m,
    \qquad
    |v_{n-1}(\tau)|
    \le A_2(\p\Lcal^{(\rho_2)})(\rho_2\mu)^{n-1}.
\]
Taking absolute values in \eqref{eq:g-contour-collar} proves
\eqref{eq:Grunsky-decay}, with a constant independent of $m,n$.

\emph{Absolute and uniform convergence and boundary identities.}
On the closed set
$C_{\rho_3}:=\{w\in\widehat{\CC}:|V(w)|\ge \rho_3\mu\}$, the lower bound
in \eqref{eq:Walsh-un-bound} gives
\begin{equation}
    |v_k(w)|\ge A_1(C_{\rho_3})(\rho_3\mu)^k,
    \qquad w\in C_{\rho_3},\quad k\ge0.
    \label{eq:Walsh-lower-collar}
\end{equation}
Combining this with \eqref{eq:Grunsky-decay}, for $w,\tau\in C_{\rho_3}$,
we obtain
\[
    \left|\frac{g_{mn}}{v_n(w)}\right|
    \le\frac{C(\rho_1\mu)^m}{A_1(C_{\rho_3})\rho_3\mu}
    \left(\frac{\rho_2}{\rho_3}\right)^{n-1}
\]
and
\begin{equation}
    \left|\frac{g_{mn}}{v_{m+1}(w)v_n(\tau)}\right|
    \le\frac{C}{A_1(C_{\rho_3})^2(\rho_3\mu)^2}
    \left(\frac{\rho_1}{\rho_3}\right)^m
    \left(\frac{\rho_2}{\rho_3}\right)^{n-1}.
    \label{eq:kernel-Mtest-bound}
\end{equation}
Since $\rho_1/\rho_3,\rho_2/\rho_3<1$, these bounds are summable in $n$ for each
fixed $m$, and in $(m,n)$, respectively. Since these bounds are
independent of the evaluation points, the sums of the supremum norms
are finite on the open neighborhood $\Lcal^{(\rho_3)}$ of $\p\Lcal$
and on its Cartesian square. The Weierstrass test therefore proves
the asserted absolute and uniform convergence.

The identity theorem now extends \eqref{eq:FW-Grunsky-expansion} from
$\Lcal$ to $\Lcal^{(\rho_3)}$. For the kernel identity, Taylor expansion at $\tau$ gives
\[
\begin{aligned}
    \Psi(w)-\Psi(\tau)
    &=\Psi'(\tau)(w-\tau)
      +\frac12\Psi''(\tau)(w-\tau)^2+O((w-\tau)^3),\\
    \Psi'(w)
    &=\Psi'(\tau)+\Psi''(\tau)(w-\tau)+O((w-\tau)^2).
\end{aligned}
\]
Since $\Psi'(\tau)\neq0$, dividing these expansions and subtracting
$1/(w-\tau)$ yields
\[
    \frac{\Psi'(w)}{\Psi(w)-\Psi(\tau)}-\frac{1}{w-\tau}
    =\frac{\Psi''(\tau)}{2\Psi'(\tau)}+O(w-\tau)
    \qquad(w\to\tau).
\]
To make joint regularity explicit, we define the divided difference
\[
    \Theta(w,\tau)=
    \begin{cases}
    \dfrac{\Psi(w)-\Psi(\tau)}{w-\tau},&w\ne\tau,\\[4pt]
    \Psi'(\tau),&w=\tau.
    \end{cases}
\]
It is jointly holomorphic at finite points of the product neighborhood
and is nonzero by injectivity and the nonvanishing derivative.
Consequently, the regularized kernel
\[
    \mathscr R(w,\tau):=
    \frac{\p_w\Theta(w,\tau)}{\Theta(w,\tau)}
\]
is jointly holomorphic there. Off the diagonal, it equals the left-hand
side of \eqref{eq:WG-kernel}, and on the diagonal, its value is
$\Psi''(\tau)/(2\Psi'(\tau))$.
The identity theorem extends \eqref{eq:WG-kernel} to the product
neighborhood, including the diagonal. Thus, $\mathscr R$ is jointly
continuous and bounded on the compact boundary product.
Absolute and uniform convergence justifies termwise boundary
integration. Cauchy estimates on a smaller neighborhood then justify
termwise differentiation.
\end{proof}

No equality or rationality of the Walsh weights is used here:
\eqref{eq:Walsh-un-bound} holds for an admissible sequence with arbitrary
positive weights. The admissible collar and its constants depend on the
fixed geometry, including the weights, but the choices
$\rho_0<\rho_1<\rho_2<\rho_3<1$ impose no additional arithmetic
restriction on them. The estimates involve only $|V|$ and the
polynomials $v_n$, so no global branch of the fractional powers is needed.

We now use \eqref{eq:WG-kernel} to derive the boundary matrices for the
conductivity problem. We denote by $V_1,\ldots,V_N$ the open bounded
domains enclosed by the components of $\p\Lcal$, labeled compatibly
with the physical components so that $b_j\in V_j$. We write
$\mathbb D:=\{\zeta\in\CC:|\zeta|<1\}$ for the open unit disk and choose
Riemann maps
\begin{equation}
    \chi_j:\mathbb D\longrightarrow V_j,
    \qquad
    \chi_j(0)=b_j,
    \qquad
    \chi_j'(0)>0,
    \qquad j=1,\ldots,N.
    \label{eq:chi}
\end{equation}
Under the boundary regularity assumed above, the maps $\chi_j$ extend
to the unit circle.  We use them only to parameterize the canonical
boundary:
\begin{equation}
    w_j(t):=\chi_j(e^{\mathrm{i}t}),
    \qquad
    z_j(t):=\Psi(w_j(t)),
    \qquad 0\le t<2\pi.
    \label{eq:boundaryparam}
\end{equation}
We set
\begin{equation}
    h_j(t):=|z_j'(t)|
    =
    |\Psi'(w_j(t))w_j'(t)|
    \label{eq:hj}
\end{equation}
so that $ds=h_j(t)\,dt$ on $\Gamma_j$. We choose
$s_j\in\{-1,1\}$ so that
\begin{equation}
    \nu_j(t)
    =
    -\mathrm{i}s_j\frac{z_j'(t)}{|z_j'(t)|}
    \label{eq:normal}
\end{equation}
points outward from $D_j$. For the counterclockwise disk
parameterization above, $s_j=1$. We retain $s_j$ to cover reversed
parameterizations, for which $s_j=-1$.

By Proposition~\ref{prop:material-invertibility} and
Lemma~\ref{lem:zerocharge}, the densities and their NP images have zero
charge on each component. We therefore use the nonzero Fourier modes:
\begin{equation}
    \varphi_j(z_j(t))
    =
    \frac{1}{h_j(t)}
    \sum_{k\neq0}c_{j,k}e^{\mathrm{i}kt}.
    \label{eq:zeta}
\end{equation}
By Fourier orthogonality,
\begin{equation}
    c_{j,k}
    =
    \frac{1}{2\pi}
    \int_0^{2\pi}
    \varphi_j(z_j(t))h_j(t)e^{-\mathrm{i}kt}\,dt,
    \qquad k\neq0.
    \label{eq:coeffextract}
\end{equation}

We index boundary modes by
$\mathcal I=\{(j,k):1\le j\le N,\ k\in\ZZ\setminus\{0\}\}$.
With cutoffs $1\le|k|\le M$ and $1\le m,n\le J$, the sections have
sizes $2NM\times J$ for moment matrices, $J\times2NM$ for $L$ and
source matrices, and $2NM\times2NM$ for boundary operators.
Unless a finite section is explicitly indicated, the identities below
concern the full infinite arrays.

For $j\in\{1,\ldots,N\}$, $k\in\ZZ\setminus\{0\}$, and $m,n\ge1$,
we define two matrices depending only on the canonical lemniscatic data,
with boundary Fourier modes indexing rows and Walsh modes indexing columns:
\begin{equation}
    [\mathcal M_+]_{(j,k),n} := \int_0^{2\pi}v_n(w_j(t))e^{\mathrm{i}kt}\,dt,\qquad [\mathcal M_-]_{(j,k),m} := \int_0^{2\pi}\frac{e^{\mathrm{i}kt}}{v_m(w_j(t))}\,dt. \label{eq:Tplus}
\end{equation}
The subscripts $+$ and $-$ distinguish the Walsh polynomials from
the reciprocal Walsh modes, with both matrices including all $k\neq0$.
We also define the canonical matrix
\begin{equation}
    [L]_{n,(j,k)}
    :=
    -\frac{\mathrm{i}s_j}{2\pi}
    \int_0^{2\pi}
    \frac{w_j'(t)e^{-\mathrm{i}kt}}
         {v_{n+1}(w_j(t))}\,dt,
    \qquad n\ge1.
    \label{eq:Lmatrix}
\end{equation}
The FW moment matrix is defined by
\begin{equation}
    [\mathcal M]_{(j,k),n}
    :=
    \int_0^{2\pi}
    F_n(z_j(t))e^{\mathrm{i}kt}\,dt.
    \label{eq:Tdef}
\end{equation}
For the source matrices, we use the incident mode as the row index and
the boundary Fourier mode as the column index:
\begin{align}
    [B^{(1)}]_{m,(j,k)}
    &:=
    \frac{1}{2\pi}
    \int_0^{2\pi}\p_{\nu_j}F_m(z_j(t))h_j(t)e^{-\mathrm{i}kt}\,dt,
    \label{eq:Bplus}\\
    [B^{(2)}]_{m,(j,k)}
    &:=
    \frac{1}{2\pi}
    \int_0^{2\pi}\p_{\nu_j}\overline{F_m(z_j(t))}h_j(t)e^{-\mathrm{i}kt}\,dt.
    \label{eq:Bminus}
\end{align}
With $\nu_j(z)$ regarded as a complex number, the chain rule gives
\[
    \p_{\nu_j}F_m(z)
    =\left.\frac{d}{d\varepsilon}
      F_m\bigl(z+\varepsilon\nu_j(z)\bigr)\right|_{\varepsilon=0}
    =F_m'(z)\nu_j(z).
\]
Taking complex conjugates gives
\[
    \p_{\nu_j}\overline{F_m}(z)
    =\overline{\p_{\nu_j}F_m(z)}
    =\overline{F_m'(z)}\,\overline{\nu_j(z)}.
\]

We use the complete matrices $\mathcal M$, $\mathcal M_+$,
$\mathcal M_-$, and $L$, with boundary modes ordered as
\[
    (-,\Gamma_1),\ldots,(-,\Gamma_N),\quad
    (+,\Gamma_1),\ldots,(+,\Gamma_N).
\]
Within each block, we use $n=|k|\in\NN$: the first $N$ blocks have
$k=-n$, and the last $N$ blocks have $k=n$.
On the positive Fourier indices $1,2,\ldots$, we define
\[
\mathcal N:=
\begin{pmatrix}1&0&0&\cdots\\0&2&0&\cdots\\0&0&3&\cdots\\\vdots&\vdots&\vdots&\ddots\end{pmatrix},
\qquad
I:=
\begin{pmatrix}1&0&0&\cdots\\0&1&0&\cdots\\0&0&1&\cdots\\\vdots&\vdots&\vdots&\ddots\end{pmatrix}.
\]
In the negative-then-positive ordering above, we define
\begin{equation}
\mathcal D:=
\left(\begin{array}{ccc|ccc}
-s_1\mathcal N& &0&0&\cdots&0\\
 &\ddots& &\vdots& &\vdots\\
0& &-s_N\mathcal N&0&\cdots&0\\ \hline
0&\cdots&0&s_1\mathcal N& &0\\
\vdots& &\vdots& &\ddots&\\
0&\cdots&0&0& &s_N\mathcal N
\end{array}\right).
\label{eq:frequency-matrix}
\end{equation}
Its entries are
$[\mathcal D]_{(j,k),(i,k')}=s_jk\,\delta_{ij}\delta_{kk'}$,
where $\delta_{ab}$ is the Kronecker delta.
Fourier reversal exchanges the negative and positive modes of the
same component:
\begin{equation}
\mathcal R:=
\left(\begin{array}{ccc|ccc}
0&\cdots&0&I& &0\\
\vdots& &\vdots& &\ddots&\\
0&\cdots&0&0& &I\\ \hline
I& &0&0&\cdots&0\\
 &\ddots& &\vdots& &\vdots\\
0& &I&0&\cdots&0
\end{array}\right),\qquad
[\mathcal R]_{(j,k),(i,k')}=\delta_{ij}\delta_{k,-k'}.
\label{eq:frequency-reversal}
\end{equation}
Thus, $[\mathcal R\mathcal M]_{(j,k),n}=[\mathcal M]_{(j,-k),n}$
and $\mathcal R^T=\mathcal R=\mathcal R^{-1}$.
In the same ordering, the material matrix is
\begin{equation}
\boldsymbol\Lambda:=
\left(\begin{array}{ccc|ccc}
\lambda_1I& &0&0&\cdots&0\\
 &\ddots& &\vdots& &\vdots\\
0& &\lambda_NI&0&\cdots&0\\ \hline
0&\cdots&0&\lambda_1I& &0\\
\vdots& &\vdots& &\ddots&\\
0&\cdots&0&0& &\lambda_NI
\end{array}\right).
\label{eq:material-matrix}
\end{equation}
Its entries are
$[\boldsymbol\Lambda]_{(j,k),(i,k')}=\lambda_j\delta_{ij}\delta_{kk'}$.
It commutes with $\mathcal D$ and $\mathcal R$, since each material
value is constant over the modes of its component, but it generally
does not commute with the NP matrix.
The infinite frequency multiplier $\mathcal D$ is unbounded on
unweighted square-summable sequences. In the source formulas below,
we apply it to the Fourier coefficients of analytic boundary traces.

\begin{proposition}[Factorization of the moment and source matrices]
\label{prop:TB-Grunsky}
The moment matrix satisfies
\begin{equation}
    \mathcal M=\mathcal M_++\mathcal M_-G^T.
    \label{eq:T-Grunsky}
\end{equation}
The source matrices are
\begin{equation}
    B^{(1)}=\frac{1}{2\pi}\mathcal M^T\mathcal R\mathcal D,
    \qquad
    B^{(2)}=-\frac{1}{2\pi}\overline{\mathcal M}^T\mathcal D.
    \label{eq:B-from-T}
\end{equation}
Consequently, both source matrices are explicit in $G$ and the
canonical lemniscatic data.
\end{proposition}

\begin{proof}
Since $z_j(t)=\Psi(w_j(t))$, substitution of
\eqref{eq:FW-Grunsky-expansion} into \eqref{eq:Tdef} gives
\[
    [\mathcal M]_{(j,k),m}
    =[\mathcal M_+]_{(j,k),m}
    +\sum_{n=1}^{\infty}g_{mn}[\mathcal M_-]_{(j,k),n}.
\]
Termwise integration is justified by
Lemma~\ref{lem:analytic-boundary-traces}. The sum is the
$((j,k),m)$ entry of $\mathcal M_-G^T$, which proves
\eqref{eq:T-Grunsky}.

For the source matrices, the chain rule and \eqref{eq:normal} give
\[
    h_j(t)\p_{\nu_j}F_m(z_j(t))
    =-\mathrm{i}s_jF_m'(z_j(t))z_j'(t)
    =-\mathrm{i}s_j\frac{d}{dt}F_m(z_j(t)).
\]
Both $F_m(z_j(t))$ and $e^{-\mathrm{i}kt}$ are $2\pi$-periodic, so the
boundary term in integration by parts vanishes. Hence,
\[
    [B^{(1)}]_{m,(j,k)} =-\frac{\mathrm{i}s_j}{2\pi}\int_0^{2\pi} \frac{d}{dt}F_m(z_j(t))e^{-\mathrm{i}kt}\,dt =\frac{s_jk}{2\pi}\int_0^{2\pi} F_m(z_j(t))e^{-\mathrm{i}kt}\,dt =\frac{s_jk}{2\pi}[\mathcal M]_{(j,-k),m}.
\]
Right multiplication of $\mathcal M^T$ by $\mathcal R$ replaces its
column index $(j,k)$ by $(j,-k)$. The subsequent multiplication by
$\mathcal D$ multiplies column $(j,k)$ by $s_jk$, giving the first
formula in \eqref{eq:B-from-T}.

For the conjugate source, complex conjugation changes $-\mathrm{i}$ to $\mathrm{i}$,
so
\[
    h_j(t)\p_{\nu_j}\overline{F_m(z_j(t))}
    =\mathrm{i}s_j\frac{d}{dt}\overline{F_m(z_j(t))}.
\]
The same integration by parts therefore yields
\[
    [B^{(2)}]_{m,(j,k)} =-\frac{s_jk}{2\pi}\int_0^{2\pi} \overline{F_m(z_j(t))}e^{-\mathrm{i}kt}\,dt =-\frac{s_jk}{2\pi} \overline{\int_0^{2\pi}F_m(z_j(t))e^{\mathrm{i}kt}\,dt} =-\frac{s_jk}{2\pi}\overline{[\mathcal M]_{(j,k),m}}.
\]
This is the second formula in \eqref{eq:B-from-T}.
\end{proof}

It remains to express the NP matrix through the same coefficient
matrix. We use $K$ to represent $\mathcal K_\Gamma^*$ with respect to
\eqref{eq:zeta}, with the input mode indexing rows and the output mode
indexing columns. For $i,j\in\{1,\ldots,N\}$ and
$k,k'\in\ZZ\setminus\{0\}$, we define
\begin{equation}
[K]_{(j,k),(i,k')} =\frac{1}{(2\pi)^2} \int_0^{2\pi}\!\int_0^{2\pi} e^{-\mathrm{i}k't}e^{\mathrm{i}ks}h_{i}(t) \times \frac{(z_{i}(t)-z_j(s))\cdot\nu_{i}(t)}{|z_{i}(t)-z_j(s)|^2} \,ds\,dt.
    \label{eq:KWentry}
\end{equation}
The integral is understood in the principal-value sense when $i=j$.
The matrix $K$ acts on row coefficient vectors by right multiplication.
We define the canonical NP matrix $K_0$ using the same Fourier convention
on the curves $w_j(t)$. Their outward unit normals are
$-\mathrm{i}s_jw_j'(t)/|w_j'(t)|$. Thus, the canonical NP kernel, multiplied by
the output weight $|w_{i}'(t)|$, becomes
\[
    |w_{i}'(t)|\,
    \RePart\!\left[
        \frac{-\mathrm{i}s_{i}w_{i}'(t)/|w_{i}'(t)|}
             {w_{i}(t)-w_j(s)}
    \right]
    =\RePart\!\left[
        -\mathrm{i}s_{i}\frac{w_{i}'(t)}{w_{i}(t)-w_j(s)}
    \right].
\]
For the input mode $(j,k)$, the density is
$e^{\mathrm{i}ks}/|w_j'(s)|$, so its denominator cancels the source
arc-length factor $|w_j'(s)|\,ds$. Extracting output mode $(i,k')$
as in \eqref{eq:coeffextract} then gives
\begin{equation}
    [K_0]_{(j,k),(i,k')}
    :=
    \frac{1}{(2\pi)^2}
    \int_0^{2\pi}\!\int_0^{2\pi}
    e^{-\mathrm{i}k't}e^{\mathrm{i}ks}
    \RePart\!\left[-\mathrm{i}s_{i}\frac{w_{i}'(t)}{w_{i}(t)-w_j(s)}\right]
    \,ds\,dt,
    \label{eq:K0}
\end{equation}
where one factor $1/(2\pi)$ comes from the NP operator and the other
from Fourier coefficient extraction. When $i=j$, the integral is
understood in the principal-value sense, equivalently using the
continuous diagonal limit of the real kernel for these analytic curves.
It represents the NP operator on the $N$ canonical boundary
components, with input modes indexing rows and output modes indexing
columns.

\begin{theorem}[Factorization of the NP matrix]
\label{thm:K-Grunsky}
Assume that all $\Gamma_j$, $j=1,\ldots,N$, are analytic.  With the
boundary-mode ordering above,
\begin{equation}
    K=K_0+\frac{1}{4\pi}
    \left(\mathcal M_-G^T L+
    \mathcal R\overline{\left(\mathcal M_-G^T L\right)}\mathcal R\right).
    \label{eq:K-Grunsky}
\end{equation}
Hence, all dependence of the NP matrix on the noncanonical part of the
Walsh conformal map is carried by $G$, while $K_0$, $L$, and
$\mathcal M_-$ depend only on the canonical lemniscatic geometry.
\end{theorem}

\begin{proof}
For complex numbers representing planar vectors, the dot product is
$a\cdot b=\RePart(a\overline b)$. Thus,
\[
    \frac{(z_{i}-z_j)\cdot\nu_{i}}{|z_{i}-z_j|^2}
    =\RePart\!\left[\frac{\nu_{i}}{z_{i}-z_j}\right].
\]
Also, $z_{i}=\Psi\circ w_{i}$ and \eqref{eq:normal} imply
\[
    h_{i}(t)\nu_{i}(t)
    =-\mathrm{i}s_{i}\Psi'(w_{i}(t))w_{i}'(t).
\]
Applying \eqref{eq:WG-kernel} with $w=w_{i}(t)$ and
$\tau=w_j(s)$ now gives
\[
    h_{i}(t)\frac{\nu_{i}(t)}{z_{i}(t)-z_j(s)}
    =-\mathrm{i}s_{i}w_{i}'(t)
    \left[
        \frac{1}{w_{i}(t)-w_j(s)}
        +\sum_{m,n\ge1}
        \frac{g_{mn}}{v_{m+1}(w_{i}(t))v_n(w_j(s))}
    \right].
\]
The real part of the first term gives $K_0$.
We justify the diagonal limit for the correction separately.
For $i=j$, its integrand in a fixed Fourier matrix entry is
\[
    r_i(t,s):=-\mathrm{i}s_iw_i'(t)
    \mathscr R(w_i(t),w_i(s))
    e^{\mathrm{i}ks-\mathrm{i}k't}.
\]
By Lemma~\ref{lem:analytic-boundary-traces} and
\eqref{eq:kernel-Mtest-bound}, the sum of the supremum norms of
its Walsh--Grunsky summands is bounded by a constant $C_i<\infty$.
In particular, $|r_i(t,s)|\le C_i$, including the diagonal.
For $0<\varepsilon<\pi$, the omitted periodic strip
$\{(t,s):\operatorname{dist}_{2\pi}(t,s)<\varepsilon\}$ has area
$4\pi\varepsilon$, so
\[
    \left|\iint_{\operatorname{dist}_{2\pi}(t,s)<\varepsilon}
    r_i(t,s)\,ds\,dt\right|
    \le4\pi C_i\varepsilon\longrightarrow0.
\]
The same summable bound permits termwise integration before or after
this limit. The canonical real kernel has the continuous diagonal
limit specified in \eqref{eq:K0}. Hence, the limit agrees with the
principal-value definition and produces no extra diagonal or jump
term. For $i\ne j$, the boundaries are separated, and the series converges
absolutely and uniformly on their product. We may therefore integrate
term by term.

For an input mode $(j,k)$ and an output mode $(i,k')$, the
contribution of the complex correction before taking its real part is
\begin{align*}
    &\frac{1}{(2\pi)^2}\sum_{m,n\ge1}g_{mn}
    \left(\int_0^{2\pi}\frac{e^{\mathrm{i}ks}}{v_n(w_j(s))}\,ds\right)
    \left(-\mathrm{i}s_{i}\int_0^{2\pi}
    \frac{w_{i}'(t)e^{-\mathrm{i}k't}}{v_{m+1}(w_{i}(t))}\,dt\right)\\
    &=\frac{1}{2\pi}\sum_{m,n\ge1} [\mathcal M_-]_{(j,k),n}\,g_{mn}\,[L]_{m,(i,k')} =\frac{1}{2\pi} [\mathcal M_-G^T L]_{(j,k),(i,k')}.
\end{align*}
The last equality uses \eqref{eq:Tplus} and \eqref{eq:Lmatrix}.

To take the real part of the kernel, we average this contribution
with that of the conjugate kernel. Conjugating the Fourier factor
$e^{-\mathrm{i}k't}e^{\mathrm{i}ks}$ reverses both $k$ and $k'$. Thus, the conjugate
kernel contributes
\[
    \frac{1}{2\pi}
    \overline{[\mathcal M_-G^T L]_{(j,-k),(i,-k')}}
    =\frac{1}{2\pi}
    [\mathcal R\overline{\left(\mathcal M_-G^T L\right)}\mathcal R]_{(j,k),(i,k')}.
\]
Averaging the two contributions supplies the factor $1/(4\pi)$.
Adding $K_0$ proves \eqref{eq:K-Grunsky}.
\end{proof}

\begin{remark}[Why the canonical term is necessary]
\label{rem:canonical-K}
If $\Psi(w)=w$ on the lemniscatic exterior, then
$F_n(\Psi(w))=v_n(w)$ and therefore $G=0$. Nevertheless,
$K=K_0$ is generally nonzero because distinct canonical boundary
components interact when $N\ge2$.
The necessity of $K_0$ here is relative to the coefficient definition
\eqref{eq:FW-Grunsky-expansion}: $G$ records the reciprocal remainder
of $F_m\circ\Psi$, whereas $K_0$ records the NP interactions of the
canonical boundaries themselves.
One could include $K_0$ in a larger generalized Grunsky object, but
that object would not be the matrix $G$ defined here.
Keeping $K_0$ explicit distinguishes canonical interaction from
conformal deformation without asserting an obstruction to other
operator formulations, such as \cite{RadnellSchippersStaubach2020}.
\end{remark}

Under \eqref{eq:zeta}, the componentwise zero-charge Sobolev space
corresponds, with equivalent norms, to the Hilbert space of row sequences
\[
    X:=\left\{\boldsymbol c:
    \|\boldsymbol c\|_X^2
    :=\sum_{j=1}^N\sum_{k\ne0}\frac{|c_{j,k}|^2}{|k|}<\infty\right\}.
\] The inverse $(\boldsymbol\Lambda-K)^{-1}$
on this space represents the bounded resolvent
$(\Lambda_\Gamma-\mathcal K_\Gamma^*)^{-1}$ acting on row coefficient
sequences from the zero-charge boundary space through \eqref{eq:zeta}.
Indeed, if $\boldsymbol c$ is a row vector of density coefficients and
$\boldsymbol b$ contains the source coefficients, then
$\boldsymbol c(\boldsymbol\Lambda-K)=\boldsymbol b$.
Its finite-section counterpart is evaluated by a matrix solve in the
numerical section. In particular, no commutation of $\boldsymbol\Lambda$
with $K$ is assumed.

\paragraph{Fourier finite sections.}
We denote by $\Pi_M$ the orthogonal projection in $X$ onto the modes
$1\le|k|\le M$, and set $A:=\boldsymbol\Lambda-K$.
Compactness of $K$ and strong convergence of $\Pi_M$ and its adjoint
to the identity give
\[
    \varepsilon_M:=\|K-\Pi_MK\Pi_M\|_{\mathcal B(X)}
    \longrightarrow0.
\]
Here, compactness follows from the boundary operator in
Proposition~\ref{prop:material-invertibility} and the Fourier isomorphism.
For $C_A:=\|A^{-1}\|_{\mathcal B(X)}$, the extended operator
$\widetilde A_M:=\boldsymbol\Lambda-\Pi_MK\Pi_M$ is invertible
whenever $C_A\varepsilon_M<1$. A Neumann-series argument and the
resolvent identity give
\[
    \|\widetilde A_M^{-1}\|
    \le\frac{C_A}{1-C_A\varepsilon_M},
    \qquad
    \|\widetilde A_M^{-1}-A^{-1}\|
    \le\frac{C_A^2\varepsilon_M}{1-C_A\varepsilon_M},
\]
where both norms are on $X$.
Since $\boldsymbol\Lambda$ commutes with $\Pi_M$, the restriction of
$\widetilde A_M$ to $\Pi_M X$ is the finite section
$A_M=\Pi_M A\Pi_M|_{\Pi_M X}$, while $\widetilde A_M$ equals
$\boldsymbol\Lambda$ on the complement of $\Pi_M X$. Thus, $A_M$ is the
exact compression of $A$ to the retained Fourier modes.
Consequently, $A_M$ is invertible for all sufficiently large $M$, and its
inverses are uniformly bounded.
For each fixed source row $\boldsymbol b\in X$, the finite solution
$\boldsymbol c_M=(\boldsymbol b\Pi_M)A_M^{-1}$, extended by zero,
satisfies
\[
    \|\boldsymbol c_M-\boldsymbol bA^{-1}\|_X
    \le
    \frac{C_A\|\boldsymbol b(I-\Pi_M)\|_X
    +C_A^2\varepsilon_M\|\boldsymbol b\|_X}
    {1-C_A\varepsilon_M}
    \longrightarrow0.
\]
This is a qualitative result for exact matrix entries at fixed geometry
and contrast. It gives no rate in $M$ and does not control additional
Grunsky truncation, quadrature, or outgoing-truncation errors.

Before substituting the Grunsky formulas, the definitions of the source,
resolvent, and moment matrices give the exact coefficient identity
\begin{equation}
    \mathbb F^{(1)}
    =
    B^{(1)}
    (\boldsymbol\Lambda-K)^{-1}
    \mathcal M,
    \qquad
    \mathbb F^{(2)}
    =
    B^{(2)}
    (\boldsymbol\Lambda-K)^{-1}
    \mathcal M.
    \label{eq:TRB}
\end{equation}
The inverse of a finite section must be distinguished from the
corresponding section of the infinite inverse.
For each fixed incident index $m$ and $a=1,2$, the row
$B^{(a)}_{m,\cdot}$ has rapidly decreasing Fourier coefficients and
belongs to the weighted space above. Applying the bounded resolvent
gives a density row $\boldsymbol c$ in that space. For every fixed
moment index $n$, its pairing with the $n$th column of $\mathcal M$
is absolutely convergent, since
\[
    \sum_{j=1}^N\sum_{k\ne0}
    \bigl|c_{j,k}[\mathcal M]_{(j,k),n}\bigr|
    \le
    \left(\sum_{j=1}^N\sum_{k\ne0}
    \frac{|c_{j,k}|^2}{|k|}\right)^{1/2}
    \left(\sum_{j=1}^N\sum_{k\ne0}
    |k|\,|[\mathcal M]_{(j,k),n}|^2\right)^{1/2}
    <\infty.
\]
The second factor is finite because $F_n(z_j(t))$ is analytic in $t$.
Thus, \eqref{eq:TRB} and the FWPT and CGPT factorizations below are
identities for each fixed pair of incident and moment indices.
In \eqref{eq:B-from-T}, $\mathcal D$ acts on the analytic trace
coefficients to form the source row before the resolvent is applied.
The Walsh-index sums involving $G$, $\mathcal M_\pm$, and $L$ in
\eqref{eq:T-Grunsky} and \eqref{eq:K-Grunsky} are justified by the
absolute and uniform convergence in Lemma~\ref{lem:analytic-boundary-traces}.
These coefficient identities do not require the individual factors to
be bounded on unweighted sequences of Walsh coefficients.

Substituting the preceding Grunsky identities into the
source--resolvent--moment formula gives the main factorization.

\begin{theorem}[Factorization of the FWPTs]
\label{thm:FWPT-Grunsky}
Assume that all $\Gamma_j$, $j=1,\ldots,N$, are analytic.  With the
FWPT convention
$\mathbb F^{(a)}=(\mathbb F_{mn}^{(a)})_{m,n\ge1}$ for $a=1,2$,
\begin{equation}
    \mathbb F^{(1)} =\frac{1}{2\pi}\mathcal M^T\mathcal R\mathcal D (\boldsymbol\Lambda-K)^{-1}\mathcal M,\qquad \mathbb F^{(2)} =-\frac{1}{2\pi}\overline{\mathcal M}^T\mathcal D (\boldsymbol\Lambda-K)^{-1}\mathcal M. \label{eq:FWPT-Grunsky1}
\end{equation}
Here, $\mathcal M=\mathcal M_++\mathcal M_-G^T$, $K$ is given by
\eqref{eq:K-Grunsky}, and $\boldsymbol\Lambda$ by
\eqref{eq:material-matrix}. All geometric matrices are independent
of $\sigma_m,\sigma_1,\ldots,\sigma_N$.
\end{theorem}

\begin{proof}
Substitution of \eqref{eq:B-from-T} into \eqref{eq:TRB}
proves both formulas.
\end{proof}

Combining Theorem~\ref{thm:FWPT-Grunsky} with
Theorem~\ref{thm:CGPTFWPT} gives the corresponding CGPT
factorization:
\begin{align}
    \mathbb N^{(1)}
    &=\frac{1}{2\pi}P^{-1}\mathcal M^T\mathcal R\mathcal D
    (\boldsymbol\Lambda-K)^{-1}\mathcal M(P^{-1})^T,
    \label{eq:CGPT-Grunsky1}\\
    \mathbb N^{(2)}
    &=-\frac{1}{2\pi}(\overline P)^{-1}\overline{\mathcal M}^T\mathcal D
    (\boldsymbol\Lambda-K)^{-1}\mathcal M(P^{-1})^T.
    \label{eq:CGPT-Grunsky2}
\end{align}

\begin{corollary}[Simply connected reduction]
\label{cor:simply-connected}

Let $D=D_1$ be a bounded simply connected inclusion with analytic boundary,
conductivity $\sigma_1$, and background conductivity $\sigma_m$. Set
\[
\lambda=\frac{\sigma_1+\sigma_m}{2(\sigma_1-\sigma_m)},
\qquad |\lambda|>\frac12.
\]
Choose the canonical exterior $\{|w|>\mu\}$, so that
$v_n(w)=w^n$ and the FW polynomials reduce to the classical Faber
polynomials,
\[
F_m(\Psi(w))
=
w^m+\sum_{n=1}^{\infty}g_{mn}w^{-n}.
\]
Define
\begin{equation}
\widehat G
:=
\mathcal N^{-1/2}\mu^{-\mathcal N}
G\mu^{-\mathcal N}\mathcal N^{1/2},
\qquad
[\widehat G]_{mn}
=
\sqrt{\frac{n}{m}}\frac{g_{mn}}{\mu^{m+n}}.
\label{eq:single-normalized-G}
\end{equation}
Then $\widehat G$ is symmetric, and the CGPT matrices satisfy
\begin{align}
\mathbb N^{(1)}
&=
4\pi P^{-1}\mu^{\mathcal N}\mathcal N^{1/2}\widehat G
\left[
I+(1-4\lambda^2)
(4\lambda^2I-\overline{\widehat G}\widehat G)^{-1}
\right]
\mathcal N^{1/2}\mu^{\mathcal N}(P^{-1})^T,
\label{eq:single-CGPT1}\\
\mathbb N^{(2)}
&=
8\pi\lambda(\overline P)^{-1}\mu^{\mathcal N}\mathcal N^{1/2}
\left[
I+(1-4\lambda^2)
(4\lambda^2I-\overline{\widehat G}\widehat G)^{-1}
\right]
\mathcal N^{1/2}\mu^{\mathcal N}(P^{-1})^T.
\label{eq:single-CGPT2}
\end{align}
Here, $I$ acts on the positive Fourier modes.
These formulas recover \cite[Theorem~4.2]{ChoiHelsingKangLim2023},
with its conformal radius $\gamma$ and symmetrized Grunsky matrix
corresponding to $\mu$ and $\widehat G$, respectively.

\end{corollary}

\begin{proof}
Integrating \eqref{eq:WG-kernel} in $w$ from infinity, with
$v_n(w)=w^n$, gives
\[
\log\!\left(\frac{\Psi(w)-\Psi(\tau)}{w-\tau}\right)
=
-\sum_{m,n\ge1}\frac{g_{mn}}{m}w^{-m}\tau^{-n}.
\]
The left-hand side is symmetric in $w$ and $\tau$. Hence,
$g_{mn}/m=g_{nm}/n$, and therefore
$\widehat G^T=\widehat G$.

We parameterize the canonical circle by
$w(t)=\mu e^{\mathrm{i}t}$. Since
\[
\RePart\!\left[-\mathrm{i}\frac{w'(t)}{w(t)-w(s)}\right]
=\frac12,
\qquad t\ne s,
\]
the canonical NP operator vanishes on the nonzero Fourier modes, so
$K_0=0$ on the zero-charge space. Fourier orthogonality gives
\[
\mathcal M_+
=
2\pi
\begin{pmatrix}
\mu^{\mathcal N}\\
0
\end{pmatrix},
\qquad
\mathcal M_-
=
2\pi
\begin{pmatrix}
0\\
\mu^{-\mathcal N}
\end{pmatrix},
\qquad
L=
\begin{pmatrix}
\mu^{-\mathcal N}&0
\end{pmatrix}.
\]
Thus,
\[
\mathcal M
=
2\pi
\begin{pmatrix}
\mu^{\mathcal N}\\
\mu^{-\mathcal N}G^T
\end{pmatrix},
\qquad
K
=
\frac12
\begin{pmatrix}
0&
\mu^{-\mathcal N}\overline G^T\mu^{-\mathcal N}\\
\mu^{-\mathcal N}G^T\mu^{-\mathcal N}&0
\end{pmatrix}.
\]

Using \eqref{eq:single-normalized-G} and
$\widehat G^T=\widehat G$, we have
\[
\mu^{-\mathcal N}G^T\mu^{-\mathcal N}
=
\mathcal N^{-1/2}\widehat G\mathcal N^{1/2}.
\]
Substitution into the FWPT factorization
\eqref{eq:FWPT-Grunsky1}, followed by the corresponding
two-by-two block inversion, yields
\[
\mathbb F^{(1)}
=
4\pi\mu^{\mathcal N}\mathcal N^{1/2}\widehat G
\left[
I+(1-4\lambda^2)
(4\lambda^2I-\overline{\widehat G}\widehat G)^{-1}
\right]
\mathcal N^{1/2}\mu^{\mathcal N},
\]
and
\[
\mathbb F^{(2)}
=
8\pi\lambda\mu^{\mathcal N}\mathcal N^{1/2}
\left[
I+(1-4\lambda^2)
(4\lambda^2I-\overline{\widehat G}\widehat G)^{-1}
\right]
\mathcal N^{1/2}\mu^{\mathcal N}.
\]
Finally, Theorem~\ref{thm:CGPTFWPT} gives
\[
\mathbb N^{(1)}
=
P^{-1}\mathbb F^{(1)}(P^{-1})^T,
\qquad
\mathbb N^{(2)}
=
(\overline P)^{-1}\mathbb F^{(2)}(P^{-1})^T,
\]
which proves \eqref{eq:single-CGPT1}--\eqref{eq:single-CGPT2}.
\end{proof}

\paragraph{Dependence on conformal data.}
The factorization above separates canonical and noncanonical geometric
information.  The lemniscatic parameters and the Walsh sequence determine
the canonical objects $v_n$, $\mathcal M_+$, $\mathcal M_-$, $L$, and $K_0$,
with $\mathcal D$ fixed by the Fourier indices and boundary orientations
and $\mathcal R$ by index reversal,
whereas the physical map $\Psi$ determines the FW polynomials,
the triangular coefficient matrix $P$, and the Walsh--Grunsky coefficient
matrix $G$.  Thus, the forward dependence is summarized by
\begin{equation}
    (\mathcal L,\{\beta_n\},\Psi;\sigma_m,\sigma_1,\ldots,\sigma_N)
    \longrightarrow
    (P,G,K,\boldsymbol\Lambda)
    \longrightarrow
    \bigl(\mathbb F^{(1)},\mathbb F^{(2)}\bigr)
    \longrightarrow
    \bigl(\mathbb N^{(1)},\mathbb N^{(2)}\bigr).
    \label{eq:forwardchain}
\end{equation}
\section{Numerical Examples}
\label{sec:numerics}

We illustrate the FW representation with three configurations and compare
it with independent Nystr\"om solutions. The background conductivity is
$\sigma_m=1$. The incident field is
\[
H(x)=x_1\cos\theta+x_2\sin\theta.
\]
The FW fields are evaluated from \eqref{eq:GME}.

On an exterior grid $\Omega_{\mathrm{grid}}$, we measure
\begin{equation}
\begin{aligned}
E_2
&=
\frac{
\|u_{\mathrm{FW}}-u_{\mathrm{Nys}}\|_{\ell^2(\Omega_{\mathrm{grid}})}
}{
\|u_{\mathrm{Nys}}\|_{\ell^2(\Omega_{\mathrm{grid}})}
},
\\[4pt]
\mathcal E_{\log}(x)
&=
\log_{10}
\max\!\left\{
\frac{
|u_{\mathrm{FW}}(x)-u_{\mathrm{Nys}}(x)|
}{
\displaystyle
\max_{y\in\Omega_{\mathrm{grid}}}u_{\mathrm{Nys}}(y)
-
\min_{y\in\Omega_{\mathrm{grid}}}u_{\mathrm{Nys}}(y)
},
\,2^{-52}
\right\}.
\end{aligned}
\label{eq:num-spatial-error}
\end{equation}

In the first two examples, the canonical domain has equal weights
$m_j=1/N$ and is given by
\[
|\mathfrak p(w)|>\mu^N,
\qquad
\mathfrak p(w)=\prod_{j=1}^N(w-b_j).
\]
The first two comparisons use a $600\times600$ physical grid, excluding
the inclusions and the buffer
$|\mathfrak p(\Psi^{-1}(z))|<1.08\mu^N$.

\begin{table}[H]

\centering

\begin{tabular}{ccccccc}

\toprule

Case & $N$ & $M$ & $J_G$ & $J$ & $N_{\mathrm{Nys}}$ & $E_2$\\

\midrule

1 & 7 & 48 & 168 & 168 & 512 & $6.300\times10^{-8}$\\

2 & 3 & 40 & 60 & 60 & 512 & $7.568\times10^{-9}$\\

3 & 2 & 1000 & --- & 2001 & 4000 & $2.862\times10^{-5}$\\

\bottomrule

\end{tabular}

\caption{Resolutions and relative total-field discrepancies.
$M$ is the Fourier cutoff, with modes $1\le |k|\le M$ retained on each
boundary component. $J_G$ is the cutoff for the Walsh--Grunsky
expansion, $J$ is the cutoff for the outgoing Faber--Walsh expansion,
and $N_{\mathrm{Nys}}$ is the number of Nystr\"om discretization nodes
on each boundary component.}

\label{tab:num-results}

\end{table}

\paragraph{Seven inclusions with different conductivities.}

The first example has sevenfold rotationally symmetric geometry:
\[
b_j=2e^{2\pi\mathrm{i}(j-1)/7},\qquad
\mu=2(0.75)^{1/7},\qquad
\Psi(w)=w+0.002\,\frac{2^7w}{w^7-2^7}.
\]
The conductivities and incident field are
\[
(\sigma_1,\ldots,\sigma_7)=(0.35,0.6,1.5,2.5,4,7,12),
\qquad
H(x)=x_1.
\]

\begin{figure}[htbp]
\centering
\includegraphics[width=\textwidth]{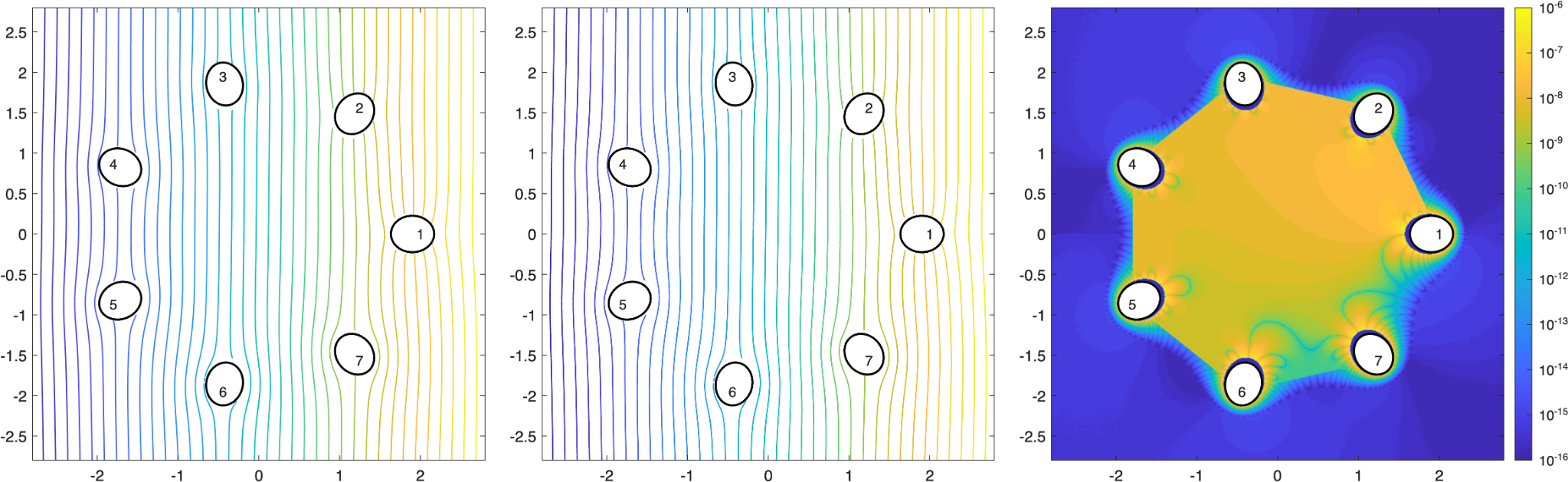}
\caption{Seven rotationally symmetric inclusions with different
conductivities. Left: Nystr\"om field. Center: FW field.
Right: the normalized pointwise error on a logarithmic scale.
The component numbers follow the ordering of $b_j$.}
\label{fig:num-case1}
\end{figure}

\FloatBarrier

\paragraph{Three asymmetric inclusions with a common conductivity.}

The second example uses
\[
(b_1,b_2,b_3)
=
(0,2.5+0.2\mathrm{i},-0.9+1.8\mathrm{i}),
\qquad
\mu=2(0.25)^{1/3},
\]
and
\[
\Psi(w)=w+\frac{a_1}{w},
\qquad
a_1=0.048e^{\mathrm{i}\pi/4},
\qquad
\sigma_1=\sigma_2=\sigma_3=10.
\]
The incident field is $H(x)=x_2$. Unlike the simply connected
case, the single nonzero Laurent coefficient does not make the
Walsh--Grunsky matrix diagonal; in the present ordering,
$g_{21}=-b_2a_1\ne0$.

\begin{figure}[htbp]
\centering
\includegraphics[width=\textwidth]{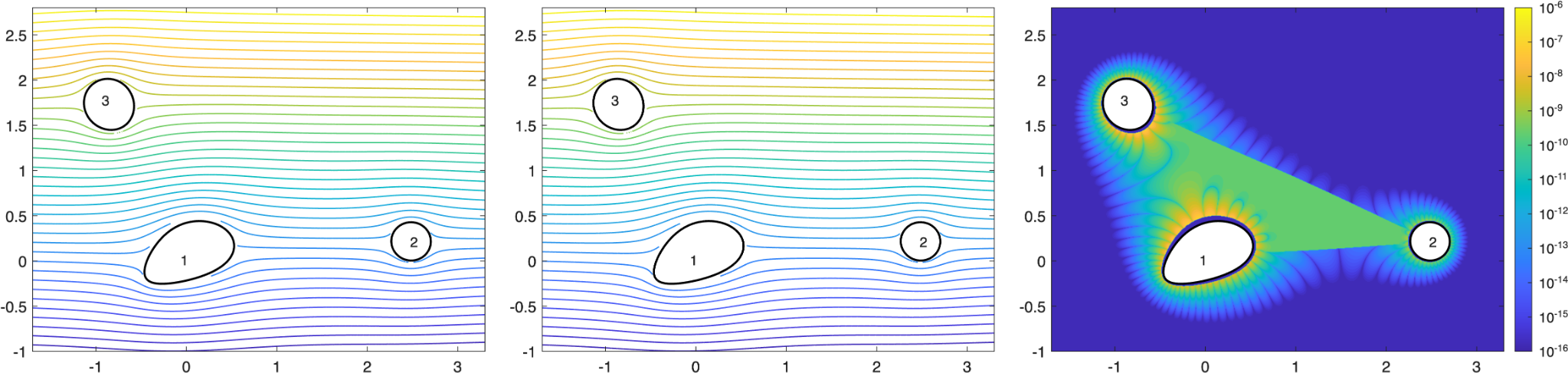}
\caption{Three asymmetric inclusions with conductivity $\sigma_j=10$
and incident field $H(x)=x_2$.
Left: Nystr\"om field. Center: FW field.
Right: the normalized pointwise error on a logarithmic scale.}
\label{fig:num-case2}
\end{figure}

\FloatBarrier

\paragraph{A nearly touching pair.}

The last example takes
\[
\Psi(w)=w,\qquad
b=1.42,\qquad
\mu=1.4195,\qquad
\sigma_1=\sigma_2=0.35.
\]
The inclusions are the bounded components of
$|w^2-b^2|<\mu^2$, with gap
$2\sqrt{b^2-\mu^2}\approx7.536\times10^{-2}$.
The incident field is
$
H(x)=x_1\cos(-30^\circ)+x_2\sin(-30^\circ).
$
Since $F_n=v_n$ and $G=0$, the NP matrix reduces to $K_0$, so this
example isolates the canonical interaction. With $M=1000$, $J=2001$,
and $N_{\mathrm{Nys}}=4000$, the relative discrepancy is
$E_2=2.862\times10^{-5}$.

\begin{figure}[htbp]
\centering
\includegraphics[width=\textwidth]{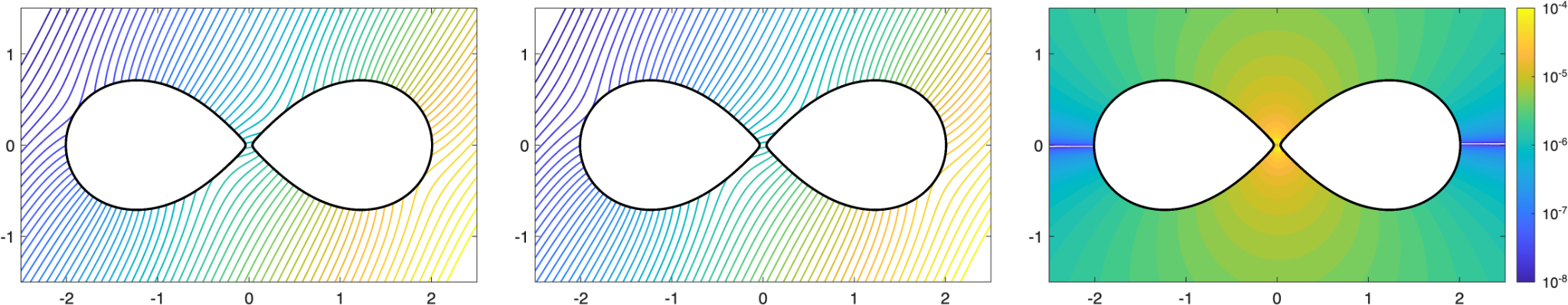}
\caption{Nearly touching inclusions with $\sigma_1=\sigma_2=0.35$.
Left: Nystr\"om field. Center: FW field.
Right: the normalized pointwise error on a logarithmic scale.}
\label{fig:num-case3}
\end{figure}

\FloatBarrier

\section{Conclusion and Future Directions}
\label{sec:conclusion}

The FW expansion gives a lemniscatic-coordinate representation of the
response of multiple conductivity inclusions with componentwise material
parameters. The factorization of the NP operator separates canonical
interaction from conformal deformation. In particular, the canonical term
need not vanish even when the Walsh--Grunsky coefficients are zero. The FWPTs
are expressed through the material resolvent and are related to the CGPTs by a
triangular change of polynomial basis. Thus, the lemniscatic coordinate gives
a representation for the field, the NP operator, and the moment
coefficients. The field expansion and the geometric truncation estimate hold
for $C^{1,\alpha}$ interfaces, whereas the NP factorization requires analytic
interfaces. The truncation estimate uses exact coefficients for fixed geometry
and contrast.

The nearly touching regime also deserves further study. The factorization
obtained here may be useful for analyzing separately the canonical interaction
and the material resolvent as the distance between components tends to zero.
Related questions arise for large conductivity contrast, where the dependence
of the coupled resolvent on the componentwise material parameters may become
especially significant. Other natural extensions include multiple inclusions
with imperfect interfaces, coated components, or less regular boundaries. In
these settings, one may ask which parts of the FW expansion, the NP
factorization, and the FWPT--CGPT relation remain valid.

The FW representation also suggests several inverse problems. Since GPTs and
CGPTs can be obtained from boundary measurements, one may ask whether finitely
many of them determine geometric and material parameters of a prescribed
family of multiple inclusions. The unknowns may include the parameters of the
lemniscatic domain, coefficients of the conformal map, and the
componentwise conductivities. It would also be interesting to determine
whether the number of components can be recovered from finite-order GPT or
CGPT data. The triangular relation between CGPTs and FWPTs gives an alternative
coordinate system for studying these questions. For electrical impedance
tomography, a further direction is to formulate the problem in a bounded
measurement domain and relate the Dirichlet-to-Neumann map to the GPT, CGPT,
and FWPT representations. This would connect the lemniscatic-coordinate
formulation directly with boundary voltage--current measurements and lead to
questions of uniqueness and stability for recovering multiple inclusions from
boundary data.


\end{document}